\documentclass[11pt, a4paper]{article}

\usepackage{amsmath}\usepackage{amsfonts}\usepackage{amssymb}\usepackage{amsthm}
\usepackage{graphicx}
\usepackage{ifthen}
\usepackage{lscape}
\usepackage{subfigure}
\usepackage{setspace}
\usepackage{subeqnarray}
\usepackage{url}

\newenvironment{keywords}{\footnotesize{\bf Keywords: }}{}
\newenvironment{AMS}{\footnotesize{\bf AMS subject classification: }}{}
\usepackage[left=1in,top=1.5in,right=1in,bottom=1in]{geometry}

\usepackage{verbatim}
\usepackage{mdwlist}
\usepackage{xcolor}
\usepackage{hyperref}
\usepackage{ulem}

\newtheorem{theorem}{Theorem}[section]
\newtheorem{lemma}[theorem]{Lemma}
\newtheorem{corollary}[theorem]{Corollary}
\newtheorem{proposition}[theorem]{Proposition}
\newtheorem{remark}{Remark}[section]
\numberwithin{equation}{section}
\numberwithin{figure}{section}

\long\def\symbolfootnote[#1]#2{\begingroup\def\thefootnote{\fnsymbol{footnote}}
\footnote[#1]{#2}\endgroup}

\author{	Alexander V.\ Shapeev\thanks{Section of Mathematics, Swiss Federal Institute of Technology (EPFL), Station 8, CH-1015, Lausanne, Switzerland 
        ({\tt alexander@shapeev.com}).
        }
}

\title{Consistent Energy-Based Atomistic/Continuum Coupling for Two-Body Potentials in One and Two Dimensions\thanks{The work was performed during the author's stay at the Chair of Computational Mathematics and Numerical Analysis (ANMC) at the Swiss Federal Institute of Technology (EPFL) whose support is acknowledged.}
}

\newcommand{\smfrac}[2]{{\textstyle\frac{#1}{#2}}}

\newcommand{\U}{{\mathcal U}}

\newcommand{\bbR}{{\mathbb R}}
\newcommand{\bbZ}{{\mathbb Z}}
\newcommand{\calI}{{\mathcal I}}
\newcommand{\calU}{{\mathcal U}}
\newcommand{\calR}{{\mathcal R}}
\newcommand{\calB}{{\mathcal B}}
\newcommand{\calT}{{\mathcal T}}
\newcommand{\calP}{{\mathcal P}}
\newcommand{\calO}{{\mathcal O}}

\newcommand{\Dc}[1]{\nabla_{\!#1}}
\newcommand{\Da}[1]{D_{\hspace{-1pt}#1}}

\def\mA{{\sf A}}

\def\mF{{\sf F}}

\newcommand{\cchi}{{\chi_{\mathstrut}}}

\def\Eqc{E^{\rm qc}}
\def\Eecc{E^{\rm ecc}}
\def\Eacc{E^{\rm acc}}
\def\Eaccmod{\tilde{E}^{\rm acc}}

\def\c{{\rm c}}
\def\a{{\rm a}}
\def\D{{\rm D}}
\def\d{{\rm d}}
\def\dx{{\rm d}x}
\def\db{{\rm db}}

\def\ea{a}
\def\eamod{\tilde{a}}
\def\ec{c}

\def\Xint#1{\mathchoice
{\XXint\displaystyle\textstyle{#1}}{\XXint\textstyle\scriptstyle{#1}}{\XXint\scriptstyle\scriptscriptstyle{#1}}{\XXint\scriptscriptstyle\scriptscriptstyle{#1}}\!\int}
\def\XXint#1#2#3{{\setbox0=\hbox{$#1{#2#3}{\int}$ }
\vcenter{\hbox{$#2#3$ }}\kern-.6\wd0}}
\def\mint{\Xint-}

\begin{document}
\sloppy
\maketitle

\begin{abstract}
This paper addresses the problem of consistent energy-based coupling of atomistic and continuum models of materials, limited to zero-temperature statics of simple crystals.
It has been widely recognized that the most practical coupled methods exhibit large errors on the atomistic/continuum interface (which are often attributed to spurious forces called ``ghost forces'').
There are only few existing works that propose a coupling which is sufficiently accurate near the interface under certain limitations.
In this paper a novel coupling that is free from ``ghost forces'' is proposed for a two-body interaction potential under the assumptions of either (i) one spatial dimension, or (ii) two spatial dimensions and piecewise affine finite elements for describing the continuum deformation.
The performance of the proposed coupling is demonstrated with numerical experiments.
The coupling strategy is based on judiciously defining the contributions of the atomistic bonds to the discrete and the continuum potential energy.
The same method in one dimension has been independently developed and analyzed in [H.~X.~Li and M.~Luskin, {\it IMA J.~Numer.~Anal.}, to appear].
\end{abstract}



\pagestyle{myheadings}
\thispagestyle{plain}
\markboth{ALEXANDER V. SHAPEEV}{CONSISTENT ATOMISTIC/CONTINUUM COUPLING}

\section{Introduction}

In many applications of solid mechanics, such as modeling cracks, structural defects, or nanoelectromechanical systems, the classical continuum description is not suitable, and it is required to utilize an atomistic description of materials.
However, full atomistic simulations are prohibitively expensive; hence there is a need for efficient numerical methods that couple a continuum description of a material in the region where material deformation is smooth and an atomistic description where the variations of the deformation gradient are large.

All atomistic/continuum (A/C) coupling methods can be divided into two categories: the energy-based coupling and the force-based coupling.
Energy-based coupling \cite{BroughtonAbrahamBernsteinEtAl1999, ELuYang2006, EidelStukowski2009, KleinZimmerman2006, MingYang2009, TadmorOrtizPhillips1996, WagnerLiu2003, XiaoBelytschko2004} consists of composing the energy of the system depending on both discrete and continuum deformations, and driving the system to a stable equilibrium.
A major challenge for energy-based A/C coupling is the presence of ghost forces---the spurious forces that create an error in the solution near the A/C interface that cannot be reduced by enlarging the atomistic region or refining the computational mesh in the continuum region.

In this paper the term {\it consistent} will be used as a synonym for the absence of ghost forces.
It has been recently proved in two dimensions (2D) under some rather general technical assumptions, that absence of ghost forces implies first-order consistency \cite{OrtnerZhang:consistency}.
In addition, in \cite{OrtnerShapeev2011} it is shown directly that the proposed method exhibits a first-order convergence rate.

In one dimension (1D) the problem of consistent zero-temperature static energy-based coupling is easy to solve \cite{ELuYang2006,LinShapeev,MingYang2009}.
However, in higher dimensions this is a true challenge.
In the case of nearest neighbor interaction there are methods that are free from ghost forces (the coupling of length scales method \cite{BroughtonAbrahamBernsteinEtAl1999} is an example of such a method), but in a general situation all existing coupling methods introduce a certain interfacial error.

Several methods have been proposed to address the problem of consistent energy-based coupling, including the quasinonlocal (QNL) quasicontinuum method \cite{ShimokawaMortensenSchiotzEtAl2004} and the geometrically consistent scheme (GCS) \cite{ELuYang2006, MingYang2009}, which may be regarded as a generalization of QNL.
The latter method consists of two steps: first, passing from general finite range interaction to a nearest neighbor interaction in a padding region between the atomistic and continuum region, and second, passing from the nearest neighbor atomistic model to the continuum model.
E, Lu, and Yang showed that the latter method exhibits no ghost forces in 1D and no ghost forces in higher dimensions in the case of a straight (planar) interface with no corners \cite{ELuYang2006}.

Another noteworthy effort to minimize the ghost forces is a work of Klein and Zimmerman \cite{KleinZimmerman2006} who considered a problem of coupling in arbitrarily overlapping atomistic and continuum regions.
Assuming a two-body interatomic potential, they proposed a method of computing the contributions of the atomistic bonds in the overlapping region by numerically minimizing the ghost forces using the least squares technique.

The analysis of energy-based methods reveals the following advantages of consistent energy-based methods in 1D: (i) their error can be efficiently controlled as opposed to a finite error of nonconsistent methods \cite{DobsonLuskin2009, MingYang2009, Ortner2010qnl1d}, and (ii) their region of stable deformations essentially coincides with the one of the atomistic model, whereas inconsistent methods significantly underpredict the critical strain \cite{DobsonLuskin2009, DobsonLuskinOrtner2009:qce.stab}.

In the force-based methods \cite{FishNuggehallyShephardEtAl2007, KnapOrtiz2001, KohlhoffSchmauder1989, LuanHyunMolinariEtAl2006, ShenoyMillerTadmorEtAl1999, ShilkrotMillerCurtin2002}, the equilibrium is achieved by computing the generalized forces for each degree of freedom (associated with the atoms and the finite element nodes) and driving them to zero.
The force-based methods are the most commonly used A/C coupling methods, as most of them exhibit no ghost forces by construction.
However, there are difficulties arising from the force field being nonconservative \cite{DobsonLuskinOrtner2010, dobs-qcf2, MillerTadmor2009}.

In summary, consistent energy-based methods have good accuracy and stability properties; however, the existing energy-based methods are consistent only in a few special cases.
The force-based coupling is the most popular alternative to the energy-based coupling, but the force-based methods are not sufficiently well understood at present.

In the present paper a new consistent energy-based A/C coupling is proposed.
The coupling is based on judiciously treating the atomic bonds near the A/C interface by consistently defining their contributions to the discrete and the continuum energy.
Two variants of the method are formulated; the first one couples the atomistic and the continuum energies through the atoms inside the continuum region, and the second one performs coupling only through the interface.
The scope of the present paper is limited to a two-body interaction potential, one or two spatial dimensions, and piecewise linear finite element discretization of the continuum deformation (the latter assumption is only required for the two-dimensional case).
No restrictions on the finite element mesh (except that its nodes are positioned at the lattice sites) are made.

We note that essentially the same method in one-dimensional setting has been independently proposed and analyzed by Li and Luskin \cite{LiLuskin}.

The paper is organized as follows: Section \ref{sec:1d_coupling} introduces the problem of A/C coupling and briefly discusses major difficulties.
In section \ref{sec:1d_method} the proposed methods are presented in detail in a one-dimensional setting and then extended to a  two-dimensional case in sections \ref{sec:manyD} and \ref{sec:2d:acc}.
The numerical experiments illustrating performance of the proposed methods are presented in section \ref{sec:num}.
The paper ends with a discussion of the results (section \ref{sec:discussion}) and concluding remarks (section \ref{sec:conclusion}).

\section{A/C coupling in 1D}\label{sec:1d_coupling}

In order to give a more vivid illustration of the proposed methods, we will start with a one-dimensional problem formulation followed by a detailed presentation of the proposed methods in 1D (section \ref{sec:1d_method}).
We will then show how to extend the proposed methods to 2D in sections \ref{sec:manyD} and \ref{sec:2d:acc}.

\subsection{Atomistic model}

Consider a one-dimensional atomistic material described by positions of atoms in the reference configuration as $x_i=i$ ($i\in\bbZ$).
For simplicity we set the lattice parameter $\epsilon=1$.
Let $\calI\subset\bbZ$ be a set of indices of atoms present in the atomistic material.
Atoms may displace from their reference positions $x_i$ to positions $y_i$, and their displacements are $u_i = y_i - x_i$.
The generic linear space containing $x_i$, $y_i$, and $u_i$ will be denoted as $\U = \{g:\calI\to\bbR\}$.
For elements of $\U$ introduce a finite difference operator $\Da{r}$ as follows:
\[
(\Da{r} g)_i = \Da{r} g_i = g_{i+r}-g_r
\quad\text{(for $r\in\bbZ$)}
.
\]
Note that if $i+r\notin\calI$, then $\Da{r} g_i$ is undefined; therefore, strictly speaking, $\Da{r}$ should not be considered as an operator $\U\to\U$.

The interaction of the atoms is described by the interaction potential $\varphi$ and the cut-off radius $R$, yielding the total interaction energy
\begin{equation}
E(y)
= \sum\limits_{\substack{i,j\in \calI\\ 1\le j-i\le R}} \varphi(y_j-y_i)
= \sum\limits_{\substack{i,i+r\in \calI\\ 1\le r\le R}} \varphi(\Da{r} y_i)
.
\label{eq:E-exact}
\end{equation}
Here and in what follows we assume that $y_j>y_i$ for $j>i$.
In practice, the proximity of atoms is measured in the physical domain, i.e., the atoms $i$ and $j$ contribute a nonzero interaction energy if $|y_j-y_i|\le \tilde{R}$.
However, for the purpose of studying consistency of the method, this treatment is essentially the same as $|j-i|\le R$, since these criteria are equivalent on the uniform deformation $y_i=F x_i$.

We consider the Dirichlet-type boundary conditions; namely, we fix the positions $y_i$ of atoms $i\in\calI_\D$ near the boundary:
\begin{equation}
y_i = F x_i
\quad
(i\in\calI_\D)
,
\label{eq:boundary_conditions}
\end{equation}
where $F$ is an arbitrary deformation tensor.
To avoid boundary effects, we require that every unconstrained atom $i$ has a full set of neighbors:
\begin{equation} \label{eq:Dirichlet-conditions-restriction}
(\forall i\in \calI\setminus\calI_\D) ~
(\forall r\in \bbZ:\, |r|\le R) ~~
i+r\in\calI.
\end{equation}
(See Figure \ref{fig:Dirichlet1D_illustration} for illustration.)

\begin{figure}
\begin{center}
\includegraphics{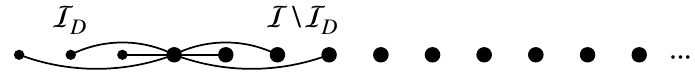}
\end{center}
\caption{An illustration of the assumption \eqref{eq:Dirichlet-conditions-restriction}: Any unconstrained atom $i\in \calI\setminus\calI_\D$ has a full set of $R$ left and $R$ right neighbors inside $\calI$.
The interaction of the leftmost unconstrained atom with the three left and the three right neighbors is displayed,
constrained atoms $i\in\calI_\D$ are shown as smaller points, and the illustration is for $R=3$.}
\label{fig:Dirichlet1D_illustration}
\end{figure}

The boundary conditions span the manifold of admissible deformations and the linear space of test functions
\[
\U_\D = \{y\in\U:\ y_i=F x_i\ \text{ for all } i\in\calI_\D \},
\qquad
\U_0 = \{u\in\U:\ u_i=0\ \text{ for all } i\in\calI_\D \}.
\]
Under the assumption \eqref{eq:Dirichlet-conditions-restriction}, the elements of $\U_0$ satisfy
\begin{equation}
\sum\limits_{i,i+r\in \calI} \Da{r} v_i = 0
\quad\text{for all}~~v\in\U_0,~r=1,\ldots,R
.
\label{eq:integral-of-derivative-is-zero}
\end{equation}

Compute the variation of $E(y)$:
\[
E'(y;v)
= \sum\limits_{\substack{i,i+r\in \calI\\ 1\le r\le R}} \varphi'(\Da{r} y_i) \Da{r} v_i.
\]
The equilibrium equations of the atomistic material under the external force $f_i$ in variational form can be written as
\begin{equation}
\textnormal{find $y\in \U_\D$:} \quad
E'(y;v) = \sum\limits_{i\in\calI} f_i v_i
\quad
\text{for all}~~v\in \U_0
.
\label{eq:exact}
\end{equation}
These equations admit the solution described by a uniform strain $y=Fx$, as is shown by the following computation:
\begin{equation}
E'(Fx;v)
= \sum\limits_{\substack{i,i+r\in \calI\\ 1\le r\le R}} \varphi'(r F) \Da{r} v_i
= \sum\limits_{1\le r\le R}  \varphi'(r F) \bigg(\sum\limits_{i,i+r\in \calI} \Da{r} v_i\bigg)
= 0
\quad \text{for all}~~v\in\U_0
,
\label{eq:variation-of-E-is-zero}
\end{equation}
which is due to \eqref{eq:integral-of-derivative-is-zero}.

The problem \eqref{eq:exact}, although discrete, is usually too large to handle on a computer.
Therefore, its approximations with reduced degrees of freedom are used.
Normally, to have a converging numerical method with the energy $\tilde{E}(y)$, one must have $\tilde{E}'(F x; v)=0$ for all $v\in\U_0$.
This relation is sometimes called the {\it ``patch test''}, a term borrowed from the theory of finite elements to describe a necessary condition for nonconforming elements to converge \cite{StrangFix1973}.
We call such approximations $\tilde{E}(y)$ {\it consistent}.
It should be noted that it has been recently proved in 2D, under rather general assumptions on the mesh, the atomistic region, and the coupling mechanism, that the absence of ghost forces implies first-order consistency \cite{OrtnerZhang:consistency}.

\subsection{Continuum model}

If the deformation gradient $y_{i+1}-y_i$ is smooth in a neighborhood of some domain $\Omega_\c$, the exact atomistic energy \eqref{eq:E-exact} can be approximated with the continuum one
\begin{equation}
E_\c(y) = \sum\limits_{r=1}^R\, \int\limits_{\Omega_\c} \varphi(\Dc{r} y) \dx,
\label{eq:E-continuum}
\end{equation}
where $\Dc{r}=r \frac{\d}{\dx}$ and $y\in W^{1,\infty}(\Omega_\c)$ is a continuum approximation to the discrete deformation (i.e., $y_i\approx y(i)$).
The formula \eqref{eq:E-continuum} is essentially the Cauchy--Born rule \cite{BornHuang1954}, where we approximate the interaction energy using the energy density $\sum_{r=1}^R \varphi(\Dc{r} y(x))$.

Compute the variation of $E_\c(y)$:
\[
E_\c'(F x, v)
= \sum\limits_{r=1}^R\, \int\limits_{\Omega_\c} \varphi'(r F) \Dc{r} v\, \dx
= \bigg(\sum\limits_{r=1}^R \varphi'(r F) r\bigg) \int\limits_{\Omega_\c} \frac{\d}{\dx}\,v(x) \dx
.
\]
If the deformation at the boundary of the continuum region is fixed, i.e., the admissible deformations $v$ satisfy $v|_{\partial\Omega_\c}=0$, then, as follows from application of the Green's formula to the above calculation, the uniform strain $y(x) = F x$ is also an equilibrium in the continuum model.

\subsection{Problem formulation}

Consider an atomistic material which in its reference configuration occupies the region $\Omega = (-N-R,N+R)$.
The material will be treated continuously in the {\itshape continuum region} $\Omega_\c = (0,N)$ and discretely in the {\itshape atomistic region} $\Omega_\a = \Omega\setminus\overline{\Omega_\c}$ (here $\overline{\mathstrut\bullet\mathstrut}$ denotes the closure of a set).
The A/C interface is $\Gamma:=\partial\Omega_\c=\{0,N\}$.
Define the atomistic lattice $\calI = \Omega\cap\bbZ$ and the atoms within the atomistic region $\calI_\a = \Omega_\a\cap\bbZ$.
The set of atoms involved in formulation of the Dirichlet-type boundary conditions \eqref{eq:boundary_conditions} is chosen as $\calI_\D = \{i\in\bbZ:\, N\le |i|<N+R\}$.
For these $\calI$ and $\calI_\D$, the relation \eqref{eq:Dirichlet-conditions-restriction} holds, and hence \eqref{eq:integral-of-derivative-is-zero} also holds.

Note that from an algorithmic point of view, fixing the positions of atoms $i\ge N+1$ (i.e., near the right boundary) is redundant, since for the continuum model in the region $(0,N)$ it is sufficient to fix the position only for the atom $i=N$.
Nevertheless, to compare the approximate model with the fully atomistic one, it is convenient to keep these atoms fixed, as we will see below.

It should also be noted that the regions were chosen in this way only for ease of visualization,
and the discussion below is valid for a more general family of regions $\Omega_\a$, $\Omega_\c$, and $\calI_\D$.
After presenting and discussing the proposed methods in 1D, we will extend the method to the two-dimensional case where we will assume a general form of the regions.

\begin{figure}
\begin{center}
\includegraphics{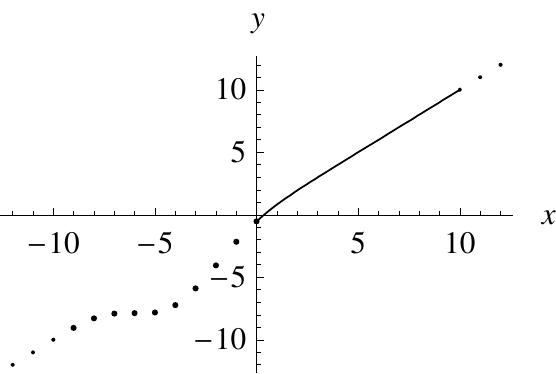}
\end{center}
\caption{A typical element $y$ of $\U$, for $N=10$ and $R=3$. The smaller points correspond to $\calI_\D$.}
\label{fig:U_element}
\end{figure}

The deformation of the material will thus be defined on $\calI_\a\cup\Gamma\cup\Omega_\c$:
\begin{align*}
\U   &= \{y: \calI_\a\cup\Omega_\c\to \bbR\,:\ y|_{\Omega_\c}\in W^{1,\infty}(\Omega_\c)\}, \\
\U_0 &= \{u\in\U:\ u_i=0\ \text{ for all } i\in\calI_\D \}.
\end{align*}
A typical element of $\U$ is illustrated in Figure \ref{fig:U_element}.
Here and in what follows we interchangeably use the notation $y_i = y(i) = y(x_i)$ for the point values of $y\in\U$.

\subsection{Quasicontinuum method}

In this subsection we formulate the energy-based quasicontinuum (QCE) method \cite{TadmorOrtizPhillips1996}, show that it has ghost forces, and describe the existing energy-based strategies of removing them.
The presence of ghost forces has been widely discussed in the literature \cite{DobsonLuskin2009, DobsonLuskinOrtner2009:qce.stab, ELuYang2006, LinShapeev, ShimokawaMortensenSchiotzEtAl2004}; hence the present discussion will be very brief and will only serve to contrast the ideas of the proposed method with the existing methods.

QCE can be defined as follows.
Write the atomistic energy in the form
\begin{equation}
\label{eq:E_preqce}
E(y) = 
		\sum\limits_{i\in\calI} \bigg(
			\frac{1}{2}
			\sum\limits_{\substack{i+r\in\calI \\ 1\le|r|\le R}} \phi(\Da{r} y_i)
		\bigg),
\end{equation}
where we define $\phi(z) = \varphi(|z|)$, and interpret the expression in the parentheses as the energy associated with a particular atom $i\in\calI$.
Then substitute the atomistic energy associated with atoms $i\notin \Omega_\a\cup\Gamma$ in the expression \eqref{eq:E_preqce} by the continuum energy:
\begin{align*}
E(y) =~& 
		\sum\limits_{i\in\calI_\a\cup\Gamma} \bigg(
			\frac{1}{2}
			\sum\limits_{\substack{i+r\in\calI \\ 1\le|r|\le R}} \phi(\Da{r} y_i)
		\bigg)
		+
		\sum\limits_{i\in\calI_\c} \bigg(
			\frac{1}{2}
			\sum\limits_{\substack{i+r\in\calI \\ 1\le|r|\le R}} \phi(\Da{r} y_i)
		\bigg)
\\ \approx~&
		\sum\limits_{i\in\calI_\a\cup\Gamma} \bigg(
			\frac{1}{2}
			\sum\limits_{\substack{i+r\in\calI \\ 1\le|r|\le R}} \phi(\Da{r} y_i)
		\bigg)
		+
		\sum\limits_{i\in\calI_\c} \bigg(
		 \frac{1}{2}
		    \int_{i-\frac12}^{i+\frac12}
			\sum\limits_{\substack{i+r\in\calI \\ 1\le|r|\le R}} \phi(\Dc{r} y) \dx
		\bigg)
\\ =~&
		\sum\limits_{i\in\calI_\a\cup\Gamma} \bigg(
			\frac{1}{2}
			\sum\limits_{\substack{i+r\in\calI \\ 1\le|r|\le R}} \phi(\Da{r} y_i)
		\bigg)
		+
		 \frac{1}{2}
		    \int_{\frac12}^{N-\frac12}
			\sum\limits_{\substack{i+r\in\calI \\ 1\le|r|\le R}} \phi(\Dc{r} y) \dx
		=: E^{\rm qce}(y)
,
\end{align*}
where the interval $(i-\smfrac12,i+\smfrac12)$ is the ``effective volume'' of an atom $i\in\calI_\c$.

To show that such a coupling has ghost forces, we let $R=2$ and show that the variation of the energy $(E^{\rm qce})'(y; v)$ for a uniform deformation $y=F x$ does not vanish.
Indeed, omitting the details of straightforward but tedious calculations, one gets
\begin{align*}
(E^{\rm qce})'(Fx; v)
=~&
	\sum\limits_{i\in\calI_\a} \bigg(
		\frac{1}{2}
		\sum\limits_{\substack{i+r\in\calI \\ r=\pm 1, \pm 2}} \phi'(r F) \Da{r} v_i
	\bigg)
	+
		 \frac{1}{2}
			\sum\limits_{\substack{i+r\in\calI \\ r=\pm 1, \pm 2}} \phi'(rF) \int_{\frac12}^{N-\frac12}\Dc{r} v \dx
\\ =~&
	\varphi'(F)\,\frac{v_1+v_0}{2}
	+
	\varphi'(2F)\,\frac{v_2+v_1+v_0+v_{-1}}{2}
	\\~& +
	\varphi'(F)\,v(\smfrac12)+\varphi'(2F)\,2v(\smfrac12)
	+ \text{(boundary terms)},
\end{align*}
where the boundary terms are concentrated near $i=N$ and are not relevant to the A/C coupling.
This can be further simplified, assuming a piecewise affine interpolation in $\Omega_\c$ (then $v(\smfrac12)=\frac{v_1+v_0}{2}$):
\begin{align*}
(E^{\rm qce})'(Fx; v)
=~&
	\varphi'(2F)\,\frac{v_2-v_1-v_0+v_{-1}}{2}
	+ \text{(boundary terms)},
\end{align*}
which is clearly nonzero.

The interpretation of presence of the nonzero (ghost) force, for instance, on atom $i=-1$ can be as follows.
The interaction of atoms $i=-1$ and $i=1$ is computed two times in a different manner: once according to the continuum strain $2 y'(x_1)$ and once according to the exact strain $y_1-y_{-1}$.
The first computation contributes the force only to the continuum region, which causes a loss of balance of forces on the atom $i=-1$.

The energy-based approach to removing the ghost forces is to modify the energy $E^{\rm qce}(y)$ in such a way that the uniform deformation $y(x) = F x$ is the solution to the equations.
In the context of the quasicontinuum method, this was first done in \cite{ShimokawaMortensenSchiotzEtAl2004} for the second nearest neighbor ($R=2$) and then generalized in \cite{ELuYang2006} to longer interactions.
However, when generalizing these methods from dimension one to higher dimensions, additional difficulties arise, for instance, related to transition between so-called element-based and atom-based summation rules near interface corners (i.e., interface edges and vertices in three dimensions) \cite{ELuYang2006}.

\section{Consistent A/C coupling in 1D}\label{sec:1d_method}

We propose a strategy of constructing a consistent coupling of the atomistic and the continuum models.
The strategy is based on splitting the energy of atomic bonds into atomistic and continuum contributions.

\subsection{Exact and continuum contributions of a bond}

We start with introducing some preliminary terms and definitions.
The term {\itshape bond} between atoms $i\in\bbZ$ and $i+r\in\bbZ$ will refer to an open interval $b = (i,i+r)$.
Introduce the set of all bonds in the atomistic system
\[
\calB := \{(i,i+r):\ 1\le r\le R,\ i\in\calI,\ i+r\in\calI \}.
\]
We denote the potential energy of a bond $b = (i,i+r)$ as
\begin{equation}
e_{(i,i+r)}(y) := \varphi(\Da{r} y_i)
\label{eq:exact-contribution-definition}
\end{equation}
and call it the {\itshape exact contribution} of the bond $(i,i+r)$ to the potential energy.
Thus, the exact atomistic potential energy \eqref{eq:E-exact} can be written as
\begin{equation}
E(y)
= \sum\limits_{b\in\calB} e_{b}(y)
.
\label{eq:E-exact-bond_contribution}
\end{equation}

For a bond $b=(i,i+r)$ fully contained in $\Omega_\c$ we define its contribution to the continuum energy or, in short, {\itshape continuum contribution}, as \[
\ec_{b}(y) := 
\frac{1}{r}\int\limits_{b} \varphi(\Dc{r} y) \dx
\qquad (\textnormal{if } b\subset\Omega_\c)
.
\]
Note that we later extend the definition of $\ec_{b}(y)$ to all bonds in \eqref{eq:continuum-contribution-definition}.

The following proposition states that the variation of the exact contribution of a bond $b$ coincides with the variation of the continuum contribution of $b$ on a uniform deformation $y=Fx$.
\begin{proposition}\label{prop:e_c_variations}
\[
\ec_{b}'(Fx; v) = e_{b}'(Fx; v)
\quad\text{for any $F>0$, $b\in\calB$, $v\in\calU$}.
\]
\end{proposition}
\begin{proof}
The validity of this proposition is verified by the following straightforward calculation:
\begin{align*}
\ec_{(i,i+r)}'(Fx; v)
= \frac{1}{r}\int\limits_{x_i}^{x_{i+r}} \varphi'(r F) r \frac{\d v}{\dx} \dx
=
\varphi'(r F) v\big|_{x_i}^{x_{i+r}}
=&~ \varphi'(r F) (v_{i+r}-v_i)
,
\\
e_{(i,i+r)}'(Fx; v)
= \varphi'(r F) \Da{r} v_i
=&~ \varphi'(r F) (v_{i+r}-v_i)
.
\end{align*}
\end{proof}

\subsection{Method of combining exact and continuum contributions}

We now define the proposed A/C coupling method motivated by Proposition \ref{prop:e_c_variations}:
\begin{equation}
\Eecc(y) := 
\sum\limits_{\substack{b\in\calB\\ b\not\subset \Omega_\c}} e_{b}(y)
+
\sum\limits_{\substack{b\in\calB\\ b\subset \Omega_\c}} \ec_{b}(y)
=: \Eecc_{\a}(y) + \Eecc_{\c}(y),
\label{eq:ecc-def}
\end{equation}
which is obtained by substituting $e_{b}(y)$ with $\ec_{b}(y)$ in \eqref{eq:E-exact-bond_contribution} for the bonds $b$ fully contained in the continuum region.
We hence name it the method of combining {\it the exact and the continuum contributions} (ECC) of the bonds (hereinafter referred to as the {\it ECC method}).

\begin{proposition}
The ECC method \eqref{eq:ecc-def} is consistent, i.e.,
\[
(\Eecc)'(Fx; v)
=
0.
\]
\end{proposition}
\begin{proof}
The consistency of \eqref{eq:ecc-def} follows directly from Proposition \ref{prop:e_c_variations} and the identity \eqref{eq:variation-of-E-is-zero}:
\[
(\Eecc)'(Fx; v)
=
\sum\limits_{\substack{b\in\calB\\ b\not\subset \Omega_\c}} e_{b}'(Fx; v)
+
\sum\limits_{\substack{b\in\calB\\ b\subset \Omega_\c}} \ec_{b}'(Fx; v)
=
\sum\limits_{b\in\calB} e_{b}'(Fx; v)
=
E'(Fx; v) = 0.
\]
\end{proof}

The remainder of this subsection will be devoted to showing that $\Eecc(y)$ can be computed efficiently, i.e., without a need to go through all the bonds $b\in\Omega_\c$.
More precisely, we will show the validity of the following proposition.
\begin{proposition}\label{prop:1d:ecc_is_efficient}
The energy of the ECC method can be written as
\begin{equation}
\label{eq:1d:ecc_is_efficient}
\Eecc(y)
=
\sum\limits_{\substack{b\in \calB\\ b\not\subset \Omega_\c}} (
	e_{b}(y) - \ec_{b}(y)
)
+
E_\c(y)
,
\end{equation}
where the definition of $\ec_b(y)$ is extended to all bonds $b$ as
\begin{equation}
\ec_{b}(y) := 
\frac{1}{r}\int\limits_{\Omega_\c\cap b} \varphi(\Dc{r} y) \dx
.
\label{eq:continuum-contribution-definition}
\end{equation}
\end{proposition}
According to this proposition, assembling the energy can be done as follows: We go through all the atoms in the atomistic region, computing the energy of their interaction with all other atoms and subtracting continuum contribution of the energy \eqref{eq:continuum-contribution-definition} for the bonds crossing the interface (i.e., for the bonds $b\not\subset\Omega_\c$ such that $\Omega_\c\cap b$ is nonempty).
The energy thus computed is then added to the continuum energy $E_\c(y)$.
Seen in this way, the method can be efficiently implemented with a single loop over the atoms within the atomistic region only.

Before we commence with a proof of Proposition \ref{prop:1d:ecc_is_efficient}, let us formulate, without a proof, the following result which is trivial in 1D but will be crucial in extending the method to 2D.
\begin{lemma}[one-dimensional bond-density lemma]
\label{lem:1d:bond_density}
Almost each point $x\in\bbR$ is covered by exactly $r$ ($r\in\bbZ^+$) bonds of the form $(i,i+r)$ ($i\in\bbZ$), i.e.,
\begin{equation}
\label{eq:1d:bond_density}
\sum\limits_{i\in\bbZ}\, \frac1r \chi_{(i,i+r)}(x) \underset{\rm a.e.}{=} 1
\quad\text{for all}~~r\in\bbZ^+
,
\end{equation}
where $\chi_\bullet(x)$ is the characteristic function of a set.
\end{lemma}

\begin{proof}[Proof of Proposition \ref{prop:1d:ecc_is_efficient}]
First, fix $r\in\bbZ^+$ and notice that, as long as \eqref{eq:Dirichlet-conditions-restriction} holds, any bond $b=(i,i+r)$ having a nonzero intersection with $\Omega_\c$ belongs to $\calB$.
Hence we can replace the sum over all $i\in\bbZ$ in \eqref{eq:1d:bond_density} by the sum over $(i,i+r)\in\calB$:
\begin{equation}
\frac{1}{r}\sum\limits_{(i,i+r)\in\calB}\chi_{(i,i+r)}(x) \underset{\rm a.e.}{=} 1
\quad\text{for all}~~x\in\Omega_\c,~r\in\bbZ^+
.
\label{eq:bond_sum_all}
\end{equation}
Hence the continuum contribution of all the bonds in the system is
\begin{equation}
\begin{array}{r@{}l} \displaystyle
\sum\limits_{b\in\calB} \ec_{b}(y)
=~&\displaystyle
\sum\limits_{b\in\calB} \frac{1}{r} \int\limits_{\Omega_\c\cap b} \varphi(\Dc{r} y) \dx
=
\sum\limits_{b\in\calB} \frac{1}{r} \int\limits_{\Omega_\c} \chi_{b}(x) \varphi(\Dc{r} y) \dx
\\ =~&\displaystyle
\int\limits_{\Omega_\c} \varphi(\Dc{r} y) \dx
=
E_\c(y)
.
\end{array}
\label{eq:continuum-energy-as-a-sum-of-bonds}
\end{equation}
Therefore we can write $\Eecc_{\c}(y)$ (cf.\ \eqref{eq:ecc-def}) in the form
\[
\Eecc_{\c}(y)
= 
\sum\limits_{\substack{b\in \calB\\ b\subset \Omega_\c}} \ec_{b}(y)
=
	E_\c(y)
-
	\sum\limits_{\substack{b\in \calB\\ b\not\subset \Omega_\c}} \ec_{b}(y),
\]
and hence \eqref{eq:1d:ecc_is_efficient} follows.
\end{proof}

\begin{remark}
The same coupling has been independently proposed and analyzed by Li and Luskin in a one-dimensional discrete setting.
In \cite{LiLuskin} they introduce an extension of the QNL method \cite{ShimokawaMortensenSchiotzEtAl2004} to the finite-range potential (i.e., arbitrary $R$), which is essentially equivalent to the ECC method of the present work,
and they analyze it in a linearized discrete setting without coarsening (i.e., for functions $y(x)$ which are piecewise affine on each interval $i-1\le x\le i$, $i=1,2,\ldots,N$).
In particular, they prove that (i) under certain assumptions the uniform deformation is a stable equilibrium, and (ii) the coupled method converges to the exact atomistic solution.
\end{remark}

\subsection{Method of combining atomistic and continuum contributions}

The method proposed in the previous subsection consists of treating each bond that does not fully lie in the continuum region atomistically and modifying the continuum energy by subtracting the corresponding contribution from the continuum energy.
Since the positions of atoms which are strictly inside $\Omega_\c$ enter the expression of the total energy, the coupling between regions becomes nonlocal.
For implementation it could be more preferable to have a coupling through the interface only.
In this subsection we present such a method.
The details of construction of this method, however, are somewhat involved and will be needed only to better understand a similar method in 2D (section \ref{sec:2d:acc}).

The calculation \eqref{eq:continuum-energy-as-a-sum-of-bonds} indicates that the method with the exact continuum energy $E_\c(y)$ should contain all the continuum contributions $\ec_{b}(y)$.
Then we should derive an {\it atomistic contribution} $\ea_{(i,i+r)}(y)$ which would consistently balance $\ec_{(i,i+r)}(y)$.

To do that, define the operator $D_{\omega} y$ for $\omega\subset\Omega$ ($\omega\ne\emptyset$) in the following way.
If $\omega = \bigcup_{m=1}^{M} (l_m, r_m)$ is a union of nonintersecting intervals, then
\[
D_{\omega} y := \sum\limits_{m=1}^{M} (y(r_m)-y(l_m)),
\]
where $|\omega|$ is the total length of $\omega$.
$D_{\omega} y$ has the following properties:
\begin{enumerate}
\item
$\frac{1}{|\omega|} D_{\omega} y$ approximates the derivative of $y$, in particular, $\frac{1}{|\omega|} D_{\omega}(F x) = F$.
\item
$D_{(i,i+r)} y = \Da{r} y_i$.
\item
If $\omega=\omega_1\cup\omega_2$ and $\omega_1\cap\omega_2=\emptyset$, then
$D_{\omega} y = D_{\omega_1} y + D_{\omega_2} y$.
\item
$
\int\limits_{\omega} v'(x) \dx = D_{\omega} v.
$
\end{enumerate}

Using this operator define the {\itshape atomistic contribution} of a bond $b=(i,i+r)$ as
\begin{equation}
\ea_{b}(y) := 
\frac{|b\cap\Omega_\a|}{r}\varphi\Big(\frac{r}{|b\cap\Omega_\a|} \,D_{b\cap\Omega_\a} y\Big).
\label{eq:atomistic-contribution-definition}
\end{equation}
Compute the variations of the atomistic and the continuum (cf.\ \eqref{eq:continuum-contribution-definition}) contributions of a bond $(i,i+r) = b\in\calB$, using properties 4 and 1 of $D_{\omega}$:
\begin{align} \notag
\ec_{b}'(Fx; v)
= \frac{1}{r} \int\limits_{b\cap\Omega_\c} \varphi'(r F) r \frac{\d v}{\dx} \dx
=& \varphi'(r F)\,D_{b\cap\Omega_\c} v
\quad\textnormal{ and }
\\ \label{eq:ea-variation}
\ea_{b}'(Fx; v)
= \frac{|b\cap\Omega_\a|}{r} \Big(\varphi'(r F)\,\frac{r}{|b\cap\Omega_\a|}\,D_{b\cap\Omega_\a} v\Big)
=& \varphi'(r F) \, D_{b\cap\Omega_\a} v,
\end{align}
and therefore, using properties 3 and 2 of $D_{\omega}$, one can see that
\[
\ec_{b}'(Fx; v) + \ea_{b}'(Fx; v)
=
\varphi'(r F)  D_{b} v
=
\varphi'(r F) \Da{r} v_i
=
e_{b}'(Fx; v).
\]
This immediately implies that if we define the energy
\[
\Eacc(y) := \sum\limits_{b\in\calB} \ea_{b}(y) + \sum\limits_{b\in\calB} \ec_{b}(y)
=: \Eacc_{\a}(y) + \Eacc_{\c}(y),
\]
then the coupling given by $\Eacc(y)$ is consistent:
\[
(\Eacc)'(Fx; v) = \sum\limits_{b\in\calB} e_{b}'(Fx; v) = E'(Fx; v) = 0.
\]
We will call it the {\it atomistic and continuum contributions} (ACC) method.

Due to \eqref{eq:continuum-energy-as-a-sum-of-bonds} the method can be written as
\begin{equation}
\Eacc(y)
=
\sum\limits_{b\in\calB} \ea_{b}(y)
+
E_{\c}(y)
.
\label{eq:pca-definition}
\end{equation}
Note that in \eqref{eq:pca-definition} the sum over $b\in\calB$ can be effectively changed to the sum over $(b\in\calB,\, b\not\subset\Omega_\c)$, since if $b\subset\Omega_\c$, then $\ea_{b}(y)=0$.

The prominent feature of the method is that the atomistic part of the energy depends only on the deformation in the atomistic region and on the interface.
This may be more convenient for implementation (see section \ref{sec:implementation}) and can potentially help in parallelization of the method.

We conclude this discussion by presenting another version of the atomistic contribution $\ea_b(y)$.
If $b\cap\Omega_\c$ is a union of nonintersecting intervals $(\xi_m, \eta_m)$ ($1\le m\le M$), then we define
\[
\eamod_{b}(y) := \sum\limits_{m=1}^M \frac{\eta_m-\xi_m}{r} \varphi\Big(r\,\frac{y_{\eta_m}-y_{\xi_m}}{\eta_m-\xi_m}\Big).
\]
Its variation
\[
\eamod_{b}'(Fx; v)
= \sum\limits_{m=1}^M \frac{\eta_m-\xi_m}{r} \varphi'(r F) r \frac{v_{\eta_m}-v_{\xi_m}}{\eta_m-\xi_m}
= \varphi'(r F) \sum\limits_{m=1}^M (v_{\eta_m}-v_{\xi_m})
\]
coincides with the variation $\ea_{b}'(Fx; v)$ (cf.\ \eqref{eq:ea-variation}), and hence $\ea_{b}(y)$ can be substituted by $\eamod_{b}(y)$ in the approximation \eqref{eq:pca-definition}:
\[
\Eaccmod(y)
=
\sum\limits_{b\in\calB} \eamod_{b}(y)
+
E_{\c}(y).
\]
The approximation $\Eaccmod(y)$ has a potentially smaller number of coupled terms when the bonds cross the interface several times (this is more likely to happen in many dimensions), which may be easier for implementation.
In the present work the ACC method was implemented for a convex atomistic region, in which case both methods coincide.

\section{Extension of ECC to 2D}\label{sec:manyD}

Coupling atomistic and continuum models of materials becomes harder in more than 1D.
For instance, the quasinonlocal quasicontinuum method \cite{ShimokawaMortensenSchiotzEtAl2004} and its generalizations \cite{ELuYang2006, MingYang2009} suffer from ghost forces near the corners of an A/C interface.

In this section we will extend the ECC method to 2D.
The method will be consistent by construction, with no restrictions on the mesh (except that its nodes coincide with lattice sites).

In the two-dimensional case the reference configuration is usually described by a uniform lattice $x_i = \mA i$
(where $i\in\bbZ^2$, $x_i\in\bbR^2$, $\mA\in\bbR^{2\times 2}$), and the deformed configuration $y_i$ is often considered as given by some mapping $Y:\bbR^2\to\bbR^2$.
However, in this paper we adopt a slightly different point of view where we set $x_i = i$ and the uniform deformation tensor $\mA$ will be accounted for as $y_i = Y(\mA x_i)$ (this is done, for instance, in \cite{Legoll2009}).
This point of view is actually closer to computer implementation, where the lattice is indexed with integers rather than real numbers.
We stress that this is not a limitation of the proposed methods; in fact the numerical examples (section \ref{sec:num}) will be presented for a hexagonal lattice with $\mA = \begin{pmatrix} 1&-1/2\\ 0 & \sqrt{3}/2 \end{pmatrix}$.

\subsection{Preliminaries}\label{sec:manyD:prelim}

Consider an open bounded region $\Omega\subset\bbR^2$, the atomistic and the continuum regions $\Omega_\a\subset\Omega$, $\Omega_\c=\Omega\setminus\overline{\Omega_\a}$ (both open), and the A/C interface $\Gamma=\partial\Omega_\c$.
We also consider the atomistic lattice $\calI = \Omega\cap\bbZ^2$, the atoms within the atomistic region $\calI_\a = \Omega_\a\cap\bbZ^2$, and the set of atoms involved in posing the Dirichlet-type boundary conditions $\calI_\D\subset\calI\cap\overline{\Omega_\a}$.
To avoid boundary effects, we will make additional assumptions on $\calI$ and $\calI_\D$ after we introduce atomistic interaction (in section \ref{sec:manyD:atomistic_continuum_model}).

We further assume that $\Omega_\c$ is a polygon with vertices coinciding with some lattice sites $x_i$.
We will consider the fully discrete case; i.e., we introduce a triangulation $\calT$ of $\Omega_\c$ with triangles $T\in\calT$ whose vertices are also positioned at the lattice sites $x_i$.
The space of continuous piecewise affine finite elements on $\Omega_\c$ is denoted as $\calP_1(\calT)$.
This setting is exactly the same as the one of the QCE method \cite{TadmorOrtizPhillips1996}.
However, the difference between the QCE and the proposed methods will be in the way the atomistic and continuum energies are coupled (cf.\ section \ref{sec:manyD:QCE_and_implementation} for more details on relation between ECC and QCE).

To formulate the proposed coupling in 2D, we introduce the following supplementary notations.
For $u:\calI\to\bbR$, $i,i+r\in\calI$, define the discrete differentiation
\[
(\Da{r} y)_i = \Da{r} y_i = y_{i+r}-y_i.
\]
Define a bond $(x_1, x_2)$ between two points $x_1, x_2\in\bbZ^2$ as an interval
\[
(x_1, x_2) := \{(1-\lambda) x_1 + \lambda x_2:\, \lambda\in(0,1)\}.
\]
Define the averaging over a bond $b = (x_1, x_2)$ of a piecewise continuous function $f:b\to\bbR$ as
\[
\mint_{(x_1, x_2)} f(x) \db := \int\limits_{0}^1 f((1-\lambda) x_1 + \lambda x_2) d\lambda.
\]
The following property then holds:
\[
\mint_{(i, i+r)} \Dc{r} f \db = \Da{r} f_i
\quad\text{for any }f\in\calP_1(\calT)
,
\]
where the directional derivative is defined as
\begin{equation}
\label{eq:Dc-definition}
\Dc{r} f(x) := \lim_{\varepsilon\to 0} (f(x+\varepsilon r)-f(x)).
\end{equation}
If $f$ is smooth, then $\Dc{r} f = r\cdot \nabla f$.
It should be noted that $\nabla f$ of a function $f\in\calP_1(\calT)$ may be undefined on a bond which fully or partly lies on an edge of some $T\in\calT$.
Nevertheless, the directional derivative $\Dc{r} f$ is piecewise continuous (more precisely, piecewise constant) on any bond $b=(i,i+r)\subset\Omega_\c$ with the same direction vector $r$.

For vector-valued functions $v$ the bond averages and the directional derivatives are defined componentwise in the same manner.

\subsection{Atomistic model and continuum approximation}\label{sec:manyD:atomistic_continuum_model}

We assume that atomistic interaction is given by a set of neighbors $\calR\subset\bbZ^2\setminus\{0\}$ and a two-body potential $\varphi$.
For simplicity of notations we denote
\begin{equation}
\label{eq:manyD:phi_through_varphi}
\phi(z) := \varphi(|z|)
\quad\text{for $z\in\bbR^2$}.
\end{equation}
Denote the collection of all bonds in the system by
\[
\calB := \{(i,i+r):\, r\in\calR,\ i\in\calI,\ i+r\in\calI\}
\]
and the {\itshape exact contribution} of a bond $b=(i,i+r)\in\calB$ under the deformation $y$ as
\begin{equation} \label{eq:manyD:exacte-definition}
e_{b}(y)
:= \phi(\Da{r} y_i).
\end{equation}
It should be understood that in this formula $r$ and $i$ essentially depend on $b$, but to simplify the notations we will avoid writing $r_b$ or $i_b$.

The energy of the atomistic model then reads
\begin{equation} \label{eq:manyD:E-definition} 
E(y) = \sum\limits_{b\in\calB} e_{b}(y),
\end{equation}
and its continuum approximation in $\Omega_\c$ based on the Cauchy--Born rule is
\begin{align}
\label{eq:manyD:Ec-definition} 
E_{\c}(y) =~&
\int\limits_{\Omega_\c} \sum\limits_{r\in\calR} \phi(\Dc{r} y) \d\Omega
\\ =~& \label{eq:manyD:Ec-definition-triangles}
\sum\limits_{T\in\calT}
\sum\limits_{r\in\calR}
	|T|\, \phi(\Dc{r} y|_T)
\quad\text{(for $y\in\calP_1(\calT)$)}
.
\end{align}

To avoid boundary effects entering our analysis, similarly to the assumption \eqref{eq:Dirichlet-conditions-restriction}, we assume
\begin{align} \label{eq:manyD:Dirichlet-conditions-restriction-1}
 \displaystyle
(\forall i\in \calI\setminus\calI_\D) ~
(\forall r\in \calR) ~~
& i+r\in\calI
\qquad \textnormal{and}
\\ \displaystyle
\label{eq:manyD:Dirichlet-conditions-restriction-2}
(\forall r\in\calR)~
\big(\forall i\in\bbZ^2:\, (i,i+r)\cap\overline{\Omega_\c}\ne\emptyset\big)~~
& i,i+r\in \calI.
\end{align}

\begin{figure}
\begin{center}
\includegraphics{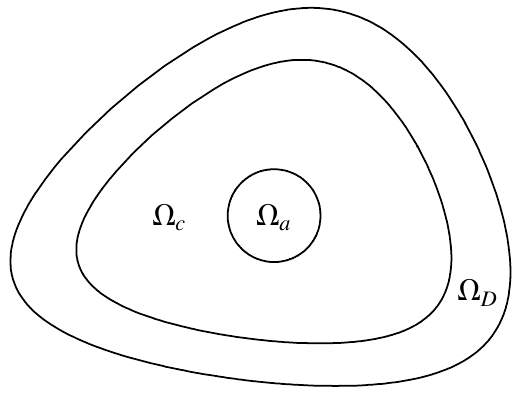}
\end{center}
\caption{An illustration of a possible choice of the regions $\Omega_\a$, $\Omega_\c$, and $\Omega_\D$. If $\Omega_\D$ is ``thick'' enough, then \eqref{eq:manyD:Dirichlet-conditions-restriction-1} and \eqref{eq:manyD:Dirichlet-conditions-restriction-2} are satisfied with $\calI_\D = \Omega_\D \cap \bbZ^2$.
}
\label{fig:regions}
\end{figure}

Practically, these assumptions mean that we have enough atoms $\calI_\D$ whose position we fix so that each free atom near the boundary of $\Omega$ (and hence each free atom in $\Omega_\c\subset\Omega$) has enough neighbors to interact with.
It can be achieved by choosing a region $\Omega_\D$ near the boundary of $\Omega$ such that $\Omega_\D\cap\Omega_\c=\emptyset$ and 
\[
{\rm dist}(\Omega\setminus\Omega_\D, \partial\Omega)>\max_{r\in\calR} |r|
\]
(see illustration on Figure \ref{fig:regions}).
Then $\calI_\D = \Omega_\D \cap \bbZ^2$ will satisfy \eqref{eq:manyD:Dirichlet-conditions-restriction-1} and \eqref{eq:manyD:Dirichlet-conditions-restriction-2}.

The spaces of deformations and displacements are then defined similarly to the one-dimensional case:
\begin{align*}
\U   &= \{y: \calI_\a\cup\overline{\Omega_\c}\to \bbR^2:~y|_{\Omega_\c}\in \calP_1(\calT)\}, \\
\U_0 &= \{u\in\U:~u_i=0\ \text{ for all } i\in\calI_\D) \}.
\end{align*}
With the assumptions \eqref{eq:manyD:Dirichlet-conditions-restriction-1} and \eqref{eq:manyD:Dirichlet-conditions-restriction-2}, one can now show that the uniform deformation $y=\mF x$ is an equilibrium of the atomistic energy \eqref{eq:manyD:E-definition}.
\begin{proposition}\label{prop:variation_of_E_is_zero}
\[
E'(\mF x; v) = 0
\quad\text{for all}~~\mF\in\bbR^{2\times 2},~v\in\calU_0
.
\]
\end{proposition}
\begin{proof}
Compute
\begin{equation}
e_{(i,i+r)}'(\mF x; v) = \phi'(\mF r) \cdot \Da{r} v_i
,
\label{eq:exact-contribution-variation-manyD}
\end{equation}
where, due to \eqref{eq:manyD:phi_through_varphi}, $\phi'(z) = \varphi'(|z|) \frac{z}{|z|}$ for $z\in\bbR^2$.
Then notice that for $v\in\U_0$ and $r\in\calR$
\[
\sum\limits_{i,i+r\in\calI} \Da{r} v_i = 0
\]
due to the assumption \eqref{eq:manyD:Dirichlet-conditions-restriction-1}.
Finally, compute
\[
E'(\mF x; v)
= \sum\limits_{(i,i+r)\in\calB} \phi'(\mF r) \cdot \Da{r} v_i
= \sum\limits_{r\in\calR} \phi'(\mF r) \cdot \sum\limits_{i,i+r\in\calI} \Da{r} v_i
= 0.
\]
\end{proof}
In the next subsection we propose a two-dimensional extension of the ECC method \eqref{eq:ecc-def} and then show that it is consistent, i.e., that its variation on the uniform deformation $y=\mF x$ is zero.

\subsection{Formulation of ECC}\label{sec:manyD:ecc_formulation}

By analogy with the one-dimensional case, define the {\itshape continuum contribution} of a bond $b\in\calB$ fully contained in $\Omega_\c$ as
\begin{equation} \label{eq:manyD:ec-predefinition}
\ec_{b}(y) := \mint_{b}
	\phi(\Dc{r} y)
	\db
\quad\text{(for $b\subset\Omega_\c$)}
,
\end{equation}
and hence define the ECC method in 2D:
\begin{equation}
\label{eq:manyD:ECC-def}
\Eecc(y) := 
\sum\limits_{\substack{b\in\calB\\ b\not\subset \Omega_\c}} e_{b}(y)
+
\sum\limits_{\substack{b\in\calB\\ b\subset \Omega_\c}} \ec_{b}(y).
\end{equation}
Note that $\ec_{b}(y)$ is well defined since $\phi(\Dc{r} y)$ is piecewise constant along each $b\subset\Omega_\c$.

\begin{proposition}\label{prop:variation_of_Eecc_is_zero}
The A/C coupling \eqref{eq:manyD:ECC-def} is consistent; i.e.,
\[
(\Eecc)'(\mF x; v) = 0
\quad\text{for all}~~\mF\in\bbR^{2\times 2},\text{~and~}v\in\calU_0
.
\]
\end{proposition}
\begin{proof}
Since
\begin{equation}
c_{(i,i+r)}'(\mF x; v) = \phi'(\mF r) \cdot \mint_{(i,i+r)} \Dc{r} v \,\db = \phi'(\mF r)\cdot \Da{r} v_i
= e_{(i,i+r)}'(\mF x; v)
,
\label{eq:continuum-contribution-variation-manyD}
\end{equation}
the variation of $\Eecc$ is zero:
\[
(\Eecc)'(\mF x; v)
=
\sum\limits_{\substack{b\in\calB\\ b\not\subset \Omega_\c}} e_{b}'(\mF x; v)
+
\sum\limits_{\substack{b\in\calB\\ b\subset \Omega_\c}} e_{b}'(\mF x; v)
=
E'(\mF x; v)
= 0,
\]
where $E'(\mF x; v) = 0$ is due to Proposition \ref{prop:variation_of_E_is_zero}.
\end{proof}

The expression \eqref{eq:manyD:ECC-def} indeed couples the energy of the system in the atomistic and the continuum region.
However, it still remains to be shown that \eqref{eq:manyD:ECC-def} can be efficiently computed.
The following theorem indeed shows this.
\begin{theorem}
\label{thm:2d:ecc_is_efficient}
The energy of the A/C coupling defined by \eqref{eq:manyD:ECC-def} can be expressed as
\begin{equation} \label{eq:manyD:Eecc}
\Eecc(y) = 
\sum\limits_{\substack{b\in\calB\\ b\not\subset \Omega_\c}} (e_{b}(y) - \ec_{b}(y))
+ E_\c(y),
\end{equation}
where the definition of $\ec_b(y)$ is extended to $b\not\subset \Omega_\c$ by \eqref{eq:manyD:ec-definition} and \eqref{eq:manyD:chi-def}
and $E_\c(y)$ is defined by \eqref{eq:manyD:Ec-definition-triangles}.
\end{theorem}

Theorem \ref{thm:2d:ecc_is_efficient} implies that the complexity of computing $\Eecc(y)$ scales as the number of bonds $b\in\calB$ such that $b\not\subset \Omega_\c$ (which is, up to a constant factor, the same as the number of atoms in $\Omega_\a$) plus the number of triangles in $\calT$.
Remarkably, the formula \eqref{eq:manyD:Eecc} is exactly the same as for the one-dimensional model \eqref{eq:1d:ecc_is_efficient} (although the respective objects are now in $\bbR^2$).

The proof of Theorem \ref{thm:2d:ecc_is_efficient} will be given in the next subsection. 

\subsection{Two-dimensional bond-density lemma}

The proof of Theorem \ref{thm:2d:ecc_is_efficient} follows the line of the proof of Proposition \ref{prop:1d:ecc_is_efficient}:
We first find a two-dimensional analogue of the bond-density lemma (Lemma \ref{lem:1d:bond_density}) and its corollary \eqref{eq:bond_sum_all}.
Based on them, we extend the definition of $\ec_b(y)$ to the bonds $b\not\subset\Omega_\c$ in a way that \eqref{eq:manyD:Eecc} holds.

A significant difficulty in extending the one-dimensional results to the two-dimensional case is that the bonds are essentially one-dimensional objects; therefore the sum of their characteristic functions equals zero almost everywhere in $\Omega_\c$.
This means that we have to look for weaker analogues of Lemma \ref{lem:1d:bond_density}.

To this end, we define the characteristic function for a polygonal set $\omega\subset\bbR^2$ in the following way:
\begin{equation}
\label{eq:manyD:chi-def}
\cchi_{\omega}(x) = \lim\limits_{\rho\to0} \frac{|\omega\cap B_\rho(x)|}{|B_\rho(x)|},
\end{equation}
where $B_\rho(x)$ is the ball with radius $\rho$ and center $x$ and $|\bullet|$ is the measure of the set.
We note that (i) the limit w.r.t.\ $\rho\to 0$ in the definition of $\cchi_{\omega}(x)$ exists, and (ii) including/excluding the boundary of a polygon $\omega$ (or any part of it) does not change the point values of $\cchi_{\omega}(x)$.

The characteristic function $\cchi_{T}(x)$ for a triangle $T$ can be visualized as follows:
\begin{equation}
\cchi_{T}(x) = \left\{
\begin{array}{lll}
1           & &\textnormal{if $x\in$ interior of $T$},
\\
\frac{1}{2} & &\textnormal{if $x\in$ edge of $T$},
\\
\frac{\alpha}{2\pi} & &\textnormal{if $x$ is a vertex of $T$ with angle $\alpha$},
\\
0 & & \textnormal{otherwise}.
\end{array}
\right.
\label{eq:chi-2d}
\end{equation}
Note that for the formulation of the method the values of $\cchi_{\omega}(x)$ at the vertices of $\omega$ will not be important.
For the characteristic function $\cchi_{\omega}(x)$ thus defined, we have
\begin{equation} \label{eq:Omegac-partition}
\cchi_{\Omega_\c}(x) = \sum\limits_{T\in\calT} \cchi_T(x)
\quad\text{for all}~~x\in\bbR^2
,
\end{equation}
where the identity is strictly pointwise (i.e., not just almost everywhere).

Using this characteristic function, it is sufficient to obtain the bond-density lemma for a single triangle $T$.
\begin{lemma}[bond-density lemma]\label{lem:area-of-triangle}
Let $T$ be a triangle in $\bbR^2$ whose vertices belong to the lattice $\bbZ^2$.
Then for any $r\in\bbZ^2$, $r\ne 0$, the following identity holds:
\begin{equation}
\sum\limits_{i\in\bbZ^2} \mint_{(i,i+r)}\cchi_T(x) \db = |T|
.
\label{eq:2d-triangle-area-through-bond-sum}
\end{equation}
\end{lemma}
\begin{proof}
Notice that if \eqref{eq:2d-triangle-area-through-bond-sum} is valid for triangles $T_1$ and $T_2$ ($T_1\cap T_2=\emptyset$), then it is valid for $T_1\cup T_2$.
Also notice that both sides of \eqref{eq:2d-triangle-area-through-bond-sum} are invariant w.r.t.\ translation of $T$ by a vector $s\in\bbZ^2$ and reflection around a point $p\in\bbZ^2$.

Using these transformations, we can obtain $mT$ (a copy of $T$ stretched by a factor $m\in\bbZ$) as a union of $m^2$ copies of $T$ shifted by different vectors and possibly reflected.
This statement can be proved by induction: for $m=1$ the statement is trivial.
Assuming that we have a partition of $mT$ by $m^2$ copies of $T$, we can complete it to a partition of $(m+1)T$ by placing $m$ parallelograms (each of which consists of two triangles) next to a side of $mT$ and one additional triangle at the corner of that side (see Figure \ref{fig:triangle-replication} for illustration).
Thus, we can partition $(m+1)T$ by $m^2 + 2m+1 = (m+1)^2$ triangles.

\begin{figure}
\begin{center}
\includegraphics{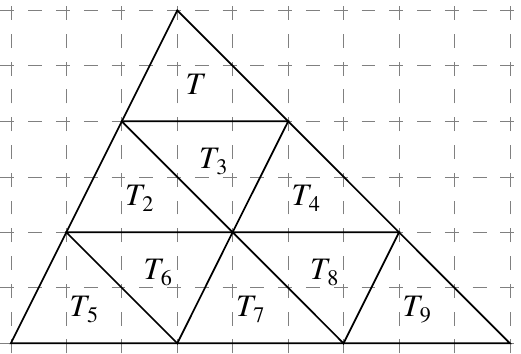}
\end{center}
\caption{An illustration of the induction step of a proof that the triangle $3T$ can be partitioned with nine copies of $T$.
	We have a partition of $2T$ with triangles $T$, $T_2$, $T_3$, $T_4$, and we complete it to a partition of $3T$ by adding two parallelograms ($T_5\cup T_6$ and $T_7\cup T_8$) and one corner triangle $T_9$.
}
\label{fig:triangle-replication}
\end{figure}

Hence one has
\begin{align*}
\sum\limits_{i\in\bbZ^2} \mint_{(i,i+r)}\cchi_T(x) \db
=~&
\frac{1}{m^{2}} \sum\limits_{i\in\bbZ^2} \mint_{(i,i+r)}\cchi_{mT}(x) \db
\displaybreak[0] \\ =~&
\frac{1}{m^{2}} \bigg(\sum\limits_{\substack{i\in\bbZ^2\\ (i,i+r)\subset mT}} ~\mint_{(i,i+r)}\cchi_{mT}(x) \db
+ \calO(m)\bigg)
\displaybreak[0] \\ =~&
\frac{1}{m^{2}} \bigg(\sum\limits_{\substack{i\in\bbZ^2 \\ (i,i+r)\subset mT}} 1
+ \calO(m)\bigg)
\displaybreak[0] \\ =~&
\frac{1}{m^{2}} \bigg(|mT|
+ \calO(m)\bigg)
=
|T|
+ \calO(m^{-1}).
\end{align*}
Here the $\calO(m)$ terms appear first from neglecting contributions of the bonds near the boundary of $mT$ (whose total measure scales as $\calO(m)$) and then from estimating the number of points $i$ for which $(i,i+r)\subset mT$ by $|mT|$, again making the error proportional to the measure of the boundary of $mT$.
Letting $m\to\infty$ finally proves \eqref{eq:2d-triangle-area-through-bond-sum}.
\end{proof}

\begin{remark}\label{rem:no3D}
Unfortunately, the three-dimensional analogue of Lemma \ref{lem:area-of-triangle} is not true.
One can take a bond direction $r=(1,0,0)$ and consider a tetrahedron $T$ with vertices $(1,0,0)$, $(0,1,0)$, $(0,0,1)$, $(0,1,1)$.
Then for each bond $(i,i+r)$ the left-hand side of \eqref{eq:2d-triangle-area-through-bond-sum} will be zero, although the right-hand side of \eqref{eq:2d-triangle-area-through-bond-sum} is $1/6$.
\end{remark}

Our next step in deriving an efficient representation of $\Eecc(y)$ consists in formulating a weak two-dimensional analogue of the identity \eqref{eq:bond_sum_all}.
\begin{lemma}\label{lem:partition-of-unity-manyD-weak}
Let $r\in\bbZ^2$, $r\ne 0$, $y\in \calP_1(\calT)$, and let $\phi$ be a continuous function $\bbR^2\to\bbR$.
Then the following identity holds:
\begin{equation}
\sum\limits_{i,i+r\in\bbZ^2} \mint_{(i,i+r)}\cchi_{\Omega_\c}\,\phi(\Dc{r} y) \db =
\int\limits_{\Omega_\c} \phi(\Dc{r} y) \d\Omega.
\label{eq:partition-of-unity-manyD-weak}
\end{equation}
\end{lemma}

\begin{proof}
The proof is based on Lemma \ref{lem:area-of-triangle} and identity \eqref{eq:Omegac-partition}.
Notice that $\Dc{r} y$ is constant in each triangle.
The following formal computation then proves \eqref{eq:partition-of-unity-manyD-weak}:
\begin{displaymath}
\begin{array}{r@{}l} \displaystyle
\sum\limits_{i,i+r\in\bbZ^2} \mint_{(i,i+r)}\cchi_{\Omega_\c}\,\phi(\Dc{r} y) \db
=~ & \displaystyle
\sum\limits_{i,i+r\in\bbZ^2} \mint_{(i,i+r)} \sum\limits_{T\in\calT} \cchi_{T}\,\phi(\Dc{r} y) \db
\\ =~ & \displaystyle
\sum\limits_{T\in\calT} \phi(\Dc{r} y)\Big|_T \bigg(\sum\limits_{i,i+r\in\bbZ^2} \mint_{(i,i+r)} \cchi_{T} \db\bigg)
\\ =~ & \displaystyle
\sum\limits_{T\in\calT} \phi(\Dc{r} y)\Big|_T \,|T|
=
\int\limits_{\Omega_\c} \phi(\Dc{r} y) \d\Omega.
\end{array}
\end{displaymath}
\end{proof}

As a final step toward deriving an efficient representation of $\Eecc(y)$, one can notice that due to the assumption \eqref{eq:manyD:Dirichlet-conditions-restriction-2}, the sum in the left-hand side of \eqref{eq:partition-of-unity-manyD-weak} can be changed to the sum over $(i,i+r)\in\calB$, which immediately yields the following corollary.

\begin{corollary}\label{corr:sum-of-bonds}
Let $r\in\calR$, $r\ne 0$, $y\in \calP_1(\calT)$, and let $f$ be a continuous function $\bbR^2\to\bbR$.
Then the following identity holds:
\begin{equation}
\sum\limits_{(i,i+r)\in\calB} \mint_{(i,i+r)}\cchi_{\Omega_\c}\,\phi(\Dc{r} y) \db =
\int\limits_{\Omega_\c} \phi(\Dc{r} y) \d\Omega.
\label{eq:sum-of-bonds}
\end{equation}
\end{corollary}

Having established the two-dimensional analogue of partition of unity (Lemma \ref{lem:partition-of-unity-manyD-weak} and Corollary \ref{corr:sum-of-bonds}), we can split the continuum energy into individual contributions.
For that, apply Corollary \ref{corr:sum-of-bonds} to the continuum energy \eqref{eq:manyD:Ec-definition}:
\begin{equation}
\label{eq:manyD:Ec-bondsum}
\begin{split}
E_\c(y)
=~&
\sum\limits_{r\in\calR}
	\int\limits_{\Omega_\c}
	\phi(\Dc{r} y)
	d\Omega
=
\sum\limits_{r\in\calR}
	\sum\limits_{(i,i+r)\in\calB}
	  \mint_{(i,i+r)}\cchi_{\Omega_\c}
		\phi(\Dc{r} y)
		\db
\\ =~&
\sum\limits_{b\in\calB}
	\mint_{b}\cchi_{\Omega_\c}
		\phi(\Dc{r} y)
		\db
=
\sum\limits_{b\in\calB}
	\ec_b(y)
,
\end{split}
\end{equation}
where we defined
\begin{equation} \label{eq:manyD:ec-definition}
\ec_{b}(y) := \mint_{b}\cchi_{\Omega_\c}
	\phi(\Dc{r} y)
	\db.
\end{equation}
\begin{proof}[Proof of Theorem \ref{thm:2d:ecc_is_efficient}]
Using \eqref{eq:manyD:ECC-def} and \eqref{eq:manyD:Ec-bondsum}, compute
\[
\Eecc(y)-E_\c(y)
=
\sum\limits_{\substack{b\in\calB\\ b\not\subset \Omega_\c}} e_{b}(y)
+
\sum\limits_{\substack{b\in\calB\\ b\subset \Omega_\c}} \ec_{b}(y)
-
\sum\limits_{b\in\calB} \ec_b(y)
=
\sum\limits_{\substack{b\in\calB\\ b\not\subset \Omega_\c}} e_{b}(y)
-
\sum\limits_{\substack{b\in\calB\\ b\not\subset \Omega_\c}} \ec_{b}(y)
,
\]
which proves \eqref{eq:manyD:Eecc}.
\end{proof}

\begin{remark}[on an alternative mesh]
As a possible modification of the above method, instead of requiring that the mesh nodes in the continuum region coincide with the lattice sites, one can require that the mesh nodes coincide with lattice half-sites (more precisely, the dual lattice sites), i.e., that the nodes of mesh triangles $T$ belong to the lattice $\bbZ^2+(1/2,1/2)$.
One can follow the proof of Lemma \ref{lem:area-of-triangle} and show that in this case it is also possible to construct a triangle $m T$ with $m^2$ copies of $T$ reflected around integer points and shifted by integer vectors.
\end{remark}

\subsection{On implementation and relation to QCE}\label{sec:manyD:QCE_and_implementation}

We can write the continuum energy of ECC using the bond-density lemma (Lemma \ref{lem:area-of-triangle}) as
\begin{equation}
\label{eq:manyD:ECC-eff-area}
\begin{split}
\Eecc_\c(y) := \sum\limits_{\substack{b\in\calB\\ b\subset \Omega_\c}} \ec_{b}(y)
=~&
\sum_{T\in\calT}
\sum\limits_{\substack{b\in\calB\\ b\subset \Omega_\c}}
\phi(\Dc{r} y|_T)
\mint_{(i,i+r)} \cchi_T \db
\\ =~&
\sum_{T\in\calT}
\sum_{r\in\calR}
\phi(\Dc{r} y|_T)
\bigg(\sum\limits_{\substack{(i,i+r)\in\calB\\ (i,i+r)\subset \Omega_\c}} \mint_{(i,i+r)} \cchi_T \db
\bigg)
\\ =~&
\sum_{T\in\calT}
\sum_{r\in\calR}
\phi(\Dc{r} y|_T)
\bigg(|T| - \sum\limits_{\substack{(i,i+r)\in\calB\\ (i,i+r)\not\subset \Omega_\c}} \mint_{(i,i+r)} \cchi_T \db\bigg)
\\ =~&
\sum_{T\in\calT}
\sum_{r\in\calR}
\Omega_{T,r} \phi(\Dc{r} y|_T),
\end{split}
\end{equation}
where
\[
\Omega_{T,r} := |T| - \sum\limits_{\substack{(i,i+r)\in\calB\\ (i,i+r)\not\subset \Omega_\c}} \mint_{(i,i+r)} \cchi_T \db
\]
is a bond-dependent effective area of $T$.
Notice that $\Omega_{T,r} = |T|$ if $T$ is distant enough from the interface $\Gamma$.
The ECC method can thus be expressed as
\begin{equation}
\label{eq:manyD:Eecc_effective_area}
\Eecc(y) = 
	\sum\limits_{\substack{b\in\calB\\ b\not\subset \Omega_\c}} e_{b}(y)
	+
	\sum_{T\in\calT}
	\sum_{r\in\calR}
	\Omega_{T,r} \phi(\Dc{r} y|_T).
\end{equation}

The QCE method uses a similar form of the continuum energy:
\[
\Eqc(y) :=
\sum\limits_{i\in\calI_\a}
\frac12\,
\bigg(
\sum\limits_{\substack{j\in\calI \\ j-i\in\calR}}
e_{(i,j)}(y)
+
\sum\limits_{\substack{j\in\calI \\ i-j\in\calR}}
e_{(i,j)}(y)
\bigg)
+
\sum_{T\in\calT}
\Omega^{\rm qc}_T
\sum_{r\in\calR}
\phi(\Dc{r} y|_T)
,
\]
where $\Omega^{\rm qc}_T$ is a particularly defined effective area of $T$, which also equals $|T|$ if $T$ is distant enough from the interface.
Thus, the ECC method can be interpreted as a modification of the QCE method by allowing the effective areas of the triangles to depend on a bond direction $r$.

The implementation of ECC is hence similar to the implementation of the QCE method for two-body interaction potentials with the only difference that one needs to compute (or precompute) the bond-dependent effective areas $\Omega_{T,r}$, which can be done as follows:
\begin{itemize}
\item[{\it Step} 1.] Set all $\Omega_{T,r} = |T|$.
\item[{\it Step} 2.] Loop over all $T\in\calT$ such that ${\rm dist}(T,\Gamma)<R$ and for each $T$ loop over all bonds $b = (i,i+r)\in\calB$ such that $b\not\subset\Omega_\c$ and $b$ intersects with the interior or an edge of $T$.
In the case if $b$ intersects with the interior of $T$, subtract $\frac{|b\cap T|}{|b|}$ from $\Omega_{T,r}$, where $|(x_1,x_2)|$ is the length of the interval between $x_1$ and $x_2$.
In the case if $b$ intersects with an edge of $T$, subtract $\frac12 \frac{|b\cap \partial T|}{|b|}$ from $\Omega_{T,r}$.
\end{itemize}
Note that the factor $\frac12$ is due to $\cchi_T(x)=\frac12$ for $x\in b\cap \partial T$.

\section{Extension of ACC to 2D}
\label{sec:2d:acc}

\begin{figure}
\begin{center}
\includegraphics{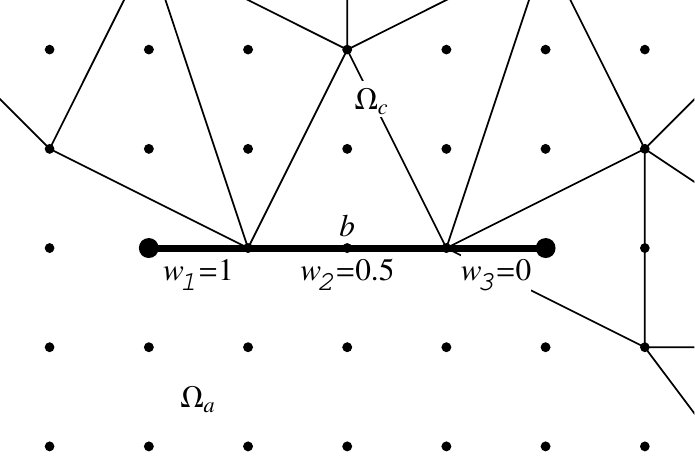}
\end{center}
\caption{An illustration of splitting a bond into several intervals in {\rm 2}D. The bond $b$ is split into three intervals, one with full contribution to the atomistic energy ($w_1=1$), the other with half contribution ($w_2=1/2$), and the third one with no contribution ($w_3=0$). The ``weights'' $w_1,w_2,w_3$ are defined in \eqref{eq:manyD:wm-definition}.}
\label{fig:bond-crosses-Omegac}
\end{figure}

In section \ref{sec:2d:acc:formulation} we will formulate another two-dimensional coupling method, which we call ACC, whose continuum energy is exactly $E_\c(y)$.
The method is a generalization of the respective one-dimensional method \eqref{eq:pca-definition}.
For its two-dimensional generalization we need to define the atomistic contribution of a bond coherent with the continuum contribution \eqref{eq:manyD:ec-definition}.

Then in section \ref{sec:implementation} we will discuss implementation of ACC.

\subsection{Formulation of ACC}
\label{sec:2d:acc:formulation}

We start with noticing that $\cchi_{\Omega_\a}(x)$ is piecewise constant on a given bond $b=(i,i+r)$.
In other words, there exist $0=t_0<t_1< \cdots<t_M=1$ such that
\begin{equation}\label{eq:manyD:wm-definition}
w_m := \cchi_{\Omega_\a}\!(i+r t)\big|_{t\in(t_{m-1},t_m)} = {\rm const}
\quad\text{for $m=1,\ldots,M$}
.
\end{equation}
The illustration of splitting a bond $b$ into several intervals is shown in Figure \ref{fig:bond-crosses-Omegac}.
Then we define the {\itshape atomistic contribution}
\begin{equation}\label{eq:manyD:ea-definition}
\ea_{b}(y)
=
\sum_{m=1}^M w_m (t_m-t_{m-1})\,
\phi\bigg(\Big(\sum_{m=1}^M w_m (t_m-t_{m-1})\Big)^{-1}\,\sum_{m=1}^M w_m (y_{i+rt_m}-y_{i+rt_{m-1}})\bigg)
\end{equation}
and the ACC method
\begin{equation} \label{eq:manyD:Eacc}
\Eacc(y) := 
\sum\limits_{b\in\calB} \ea_{b}(y)
+
\sum\limits_{b\in\calB} \ec_{b}(y)
=
\sum\limits_{\substack{b\in\calB \\ b\not\subset \Omega_\c}} \ea_{b}(y)
+
E_\c(y)
.
\end{equation}

\begin{remark}\label{rem:convex-region}
If the A/C interface $\Gamma$ is convex, then the calculation of $\ea_{b}(y)$ can be greatly simplified, since in that case not more than one section of $b$ can have a nonzero weight $w_m$, and if the weight is not equal to one, then this section of $b$ lies on the interface $\Gamma$.
\end{remark}

\begin{proposition}
The A/C coupling \eqref{eq:manyD:Eacc} is consistent.
\end{proposition}
\begin{proof}
Compute the variation of $\ea_{b}(y)$:
\begin{align*}
\ea_{b}'(\mF x; v)
=~&
\sum\limits_{m=1}^M w_m (t_m-t_{m-1})\,
	\phi'(\mF r) \cdot
	\bigg(\sum\limits_{m=1}^M w_m (t_m-t_{m-1})\bigg)^{-1}
	\sum_{m=1}^M w_m (v_{i+rt_m}-v_{i+rt_{m-1}})
\\ =~&
	\phi'(\mF r) \cdot
	\sum\limits_{m=1}^M w_m (v_{i+rt_m}-v_{i+rt_{m-1}}).
\end{align*}
Using the fact that $\Dc{r}v\,(i+r t) = \frac{\d}{\d t} v(i+r t)$ (cf.\ \eqref{eq:Dc-definition}), compute the variation of $\ec_b(y)$:
\begin{align*}
\ec_{b}'(\mF x; v)
=~& \mint_{b}
	\cchi_{\Omega_\c} \phi'(\mF r) \cdot \Dc{r}v
	\db
\\ =~&
	\int_0^1
	\cchi_{\Omega_\c}\!(i+r t) \phi'(\mF r) \cdot \Dc{r} v(i+r t) \d t
\\ =~&
	\sum\limits_{m=1}^M
	\int_{t_{m-1}}^{t_m}
	\big(1-\cchi_{\Omega_\a}\!(i+r t)\big) \phi'(\mF r) \cdot \frac{\d}{\d t} v(i+r t)) \d t
\\ =~&
	\sum\limits_{m=1}^M
	(1-w_m) \phi'(\mF r) \cdot \int_{t_{m-1}}^{t_m}
	\frac{\d}{\d t} v(i+r t)) \d t
\\ =~&
	\sum\limits_{m=1}^M
	(1-w_m) \phi'(\mF r) \cdot (v_{i+rt_m}-v_{i+rt_{m-1}})
.
\end{align*}
Observe that
\[
\ea_{b}'(\mF x; v) + \ec_{b}'(\mF x; v)
=
	\phi'(\mF r) \cdot \sum\limits_{m=1}^M
	(v_{i+rt_m}-v_{i+rt_{m-1}})
=
	\phi'(\mF r) \cdot (v_{i+r}-v_{i})
= e_{b}'(\mF x; v).
\]
Hence,
\[
(\Eacc)'(\mF x; v) = \sum\limits_{b\in\calB} e_{b}'(\mF x; v) = E'(\mF x; v) = 0
;
\]
i.e., the approximation \eqref{eq:manyD:Eacc} is consistent.
\end{proof}

\begin{remark}
The choice of atomistic contribution \eqref{eq:manyD:ea-definition} is not unique.
One can, for instance, define
\[
\eamod_{b}(y)
=
\sum\limits_{m=1}^M w_m (t_m-t_{m-1})\, \phi\Big(\frac{y_{i+rt_m}-y_{i+rt_{m-1}}}{t_m-t_{m-1}}\Big)
\]
and notice that
\begin{align*}
\eamod_{b}'(\mF x; v)
=~&
\sum\limits_{m=1}^M w_m (t_m-t_{m-1})\,
\phi'(\mF r) \cdot
\frac{v_{i+rt_m}-v_{i+rt_{m-1}}}{t_m-t_{m-1}}
\\ =~&
\phi'(\mF r) \cdot
	\sum\limits_{m=1}^M w_m (v_{i+rt_m}-v_{i+rt_{m-1}})
=\ea_{b}'(\mF x; v)
,
\end{align*}
which makes the following alternative version of \eqref{eq:manyD:Eacc},
\begin{equation} \label{eq:manyD:Eaccmod}
\Eaccmod(y) := 
\sum\limits_{b\in\calB} \eamod_{b}(y)
+
\sum\limits_{b\in\calB} \ec_{b}(y)
=
\sum\limits_{\substack{b\in\calB \\ b\not\subset \Omega_\c}} \eamod_{b}(y)
+
E_\c(y)
,
\end{equation}
consistent as well.
However, in relation to Remark \ref{rem:convex-region}, one can show that the two methods will coincide in the case of a convex atomistic region.
\end{remark}

\subsection{Notes on implementation}\label{sec:implementation}

In this section we discuss some aspects of implementation of the ACC method \eqref{eq:manyD:Eacc} in 2D.

The ACC energy \eqref{eq:manyD:Eacc} can be thought of as a sum of continuum energy $E_\c(y)$ and the respective atomistic contributions over the bonds which are not fully inside $\Omega_\c$ (i.e., bonds that are inside $\Omega_\a$ or crossing the interface $\Gamma$).
The continuum energy is treated exactly as in finite elements.
We assume a triangulation of $\overline{\Omega_\c}$ whose nodes coincide with lattice sites (this will be referred to as nodal atoms) and compute the needed quantities: energy, forces, and stiffness matrix entries; the latter is required if Newton-like methods are employed for computing equilibrium of energy.

The atomistic contributions can be split into two groups, the first one involving only bonds in $\Omega_\a$ and the second one involving bonds intersecting with $\Gamma$:
\begin{equation}\label{eq:implementation:Eacc_a-split}
\Eacc_\a(y) :=
\sum\limits_{\substack{b\in\calB \\ b\not\subset \Omega_\c}} \ea_{b}(y) = 
\sum\limits_{\substack{b\in\calB \\ \bar{b}\subset \Omega_\a}} \ea_{b}(y)
+
\sum\limits_{\substack{b\in\calB \\ \bar{b}\cap\Gamma\ne\emptyset}} \ea_{b}(y) ~=: \Eacc_{\a,1}(y) + \Eacc_{\a,2}(y)
.
\end{equation}

The first sum is nothing but the standard sum of energies of bonds for atoms in $\Omega_\a$:
\[
\Eacc_{\a,1}(y) = \sum\limits_{\substack{(i,j)\in \calB \\ (i,j)\subset\Omega_\a}} e_{(i,j)}(y)
 = \sum\limits_{\substack{(i,j)\in \calB \\ (i,j)\subset\Omega_\a}} \varphi(|y_j-y_i|)
 .
\]

\begin{figure}
\begin{center}
\includegraphics{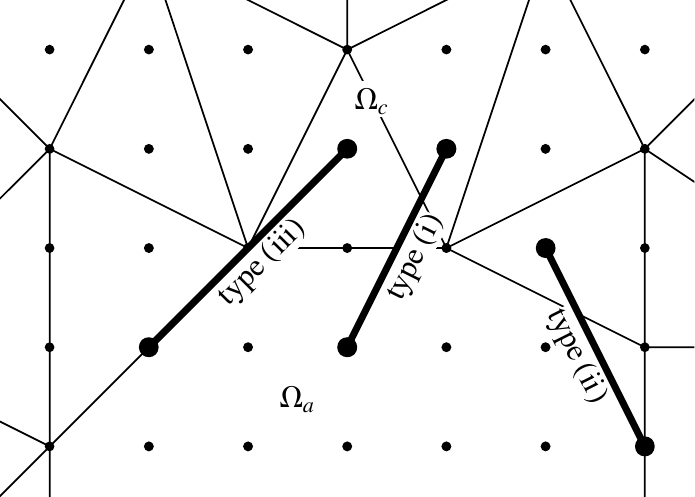}
\end{center}
\caption{An illustration of three types of bonds that should be considered when implementing ACC for convex atomistic region $\Omega_\a$: {\rm (i)} bonds that intersect $\Gamma$ at one point, {\rm (ii)} bonds that intersect $\Gamma$ at two points, and {\rm (iii)} bonds whose intersection with $\Gamma$ is an interval.}
\label{fig:convex-bonds}
\end{figure}

The second sum in \eqref{eq:implementation:Eacc_a-split} requires extra care.
In a general setting, one has to go through all the bonds: for each bond find all intersections with the interface $\Gamma$ and compute the bond's contribution according to formula \eqref{eq:manyD:ea-definition}.
According to Remark \ref{rem:convex-region}, the implementation can be simplified further if the atomistic region $\Omega_\a$ is convex (or consists of convex disjoint sets), which is indeed the case for most of the simulations of localized defects.
In this case all the bonds $b$ such that $\bar{b}\cap\Gamma\ne\emptyset$ can be further categorized into (i) bonds that intersect $\Gamma$ at one point, (ii) bonds that intersect $\Gamma$ at two points, and (iii) bonds whose intersection with $\Gamma$ is an interval (see illustration on Figure \ref{fig:convex-bonds}).

For the type-(i) bonds, we find the point of intersection of a bond $\bar{b}$ with the interface and compute the reconstruction of this point in terms of positions of nodal atoms on the interface.
Using this reconstruction, we can compute the corresponding contributions to energy, forces, and the stiffness matrix.
The type-(ii) bonds are treated in the same way, except that we need to compute reconstruction of both intersection points.
The type-(iii) bonds are treated similarly to type (ii) bonds: one should find the reconstruction of two endpoints of an intersecting interval and hence compute the contributions to energy, forces, and the stiffness matrix.
According to formula \eqref{eq:manyD:ea-definition}, one has a factor of $w_m=\frac12$ for such bonds (since $\cchi_{\Omega_\c}\big|_{\Gamma}=\frac12$, cf.\ \eqref{eq:manyD:wm-definition} and \eqref{eq:chi-2d}).
The list of coefficients of reconstruction can be precomputed once in the beginning of computation (and once per each mesh adaptation if the latter is used).
Then, at each iteration, one has to go only through this list for assembling the contributions from bonds of type (i), (ii), and (iii).

\section{Numerical experiments}\label{sec:num}

A number of numerical computations were conducted to illustrate the performance of the proposed ECC (cf.\ \eqref{eq:manyD:Eecc} or \eqref{eq:manyD:Eecc_effective_area}) and ACC (defined in \eqref{eq:manyD:Eacc}) methods in 2D and to compare them with the QCE \cite{TadmorOrtizPhillips1996}.

In all the numerical examples either the Lennard--Jones potential
\begin{equation}
\label{eq:LJ_potential}
\varphi(z) = -2 z^{-6} + z^{-12}
\end{equation}
or the Morse potential
\begin{equation}
\label{eq:Morse_potential}
\varphi(z) = -2e^{-\alpha (z-1)} + e^{-2\alpha (z-1)}
\end{equation}
was used.
With these potentials, the hexagonal lattice forms a stable equilibrium.

\begin{figure}
\begin{center}
\hfill
\subfigure[Triangulation.]{\label{fig:comparison_illustration:reference:full}
	\includegraphics{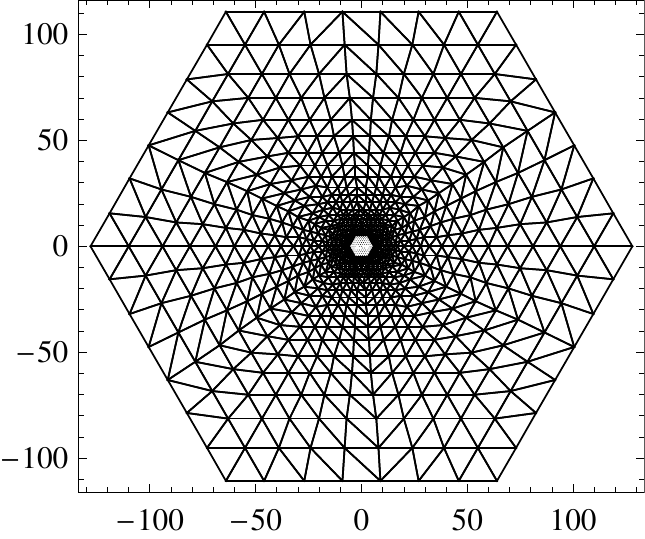}
}
\hfill \hfill
\subfigure[Closeup of the atomistic region.]{\label{fig:comparison_illustration:reference:closeup}
	\includegraphics{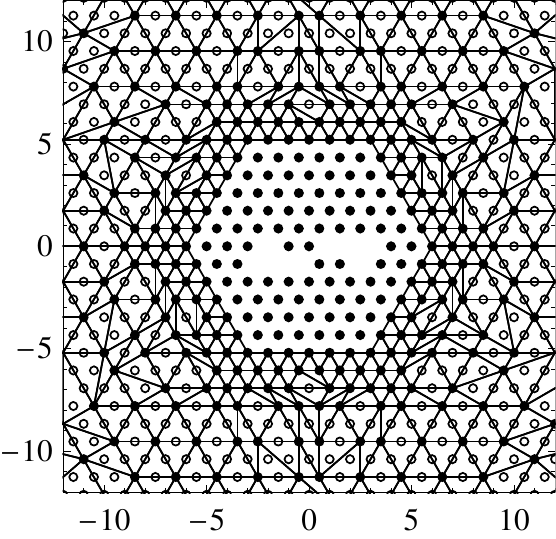}
}
\hfill $\mathstrut$
\end{center}
\caption{A reference configuration of the system of {\rm 49529} atoms forming a hexagon with the side of {\rm 129} atoms. The triangulation is shown on the left, and the closeup of the atomistic region is shown on the right.
Black atoms correspond to degrees of freedom of the system, while white atoms are kinematically ``slaved'' to the black atoms.
Eight atoms have been removed to form a defect.
The mesh is fully refined near the interface.
The illustration is for $K=6$.
}
\label{fig:comparison_illustration:reference}
\end{figure}

The test problem was chosen as follows.
We took a hexagonal atomic crystal; each side of the hexagon contains 129 atoms, and the total number of atoms in the system is 49529.
The reference atomistic configuration is illustrated in Figure \ref{fig:comparison_illustration:reference}.
Eight atoms has been removed from a perfect crystal to form a defect in its center.
The atomistic regions formed a smaller hexagon also centered at the origin whose side contained $K$ atoms, as illustrated in Figure \ref{fig:comparison_illustration:reference:closeup} for $K=6$.
The mesh was fully refined near the A/C interface for the ease of implementation of QCE.

Dirichlet-type boundary conditions were set: we extended the size of the hexagonal region by 3 atomic spacings and fixed the positions of the three added layers of atoms so that every free atom in the system had the full set of neighbors to interact with; this is done in accordance with assumptions \eqref{eq:manyD:Dirichlet-conditions-restriction-1} and \eqref{eq:manyD:Dirichlet-conditions-restriction-2}.
The ``external'' deformation gradient $\mF = \begin{pmatrix} 1&0\\ 0&0.97 \end{pmatrix}$ was applied to the positions of the added boundary atoms.
Such an external compression makes atoms occupy the empty lattice sites and form a defect, as illustrated in Figure \ref{fig:comparison_illustration:ecc6}.

\begin{figure}
\begin{center}
\includegraphics{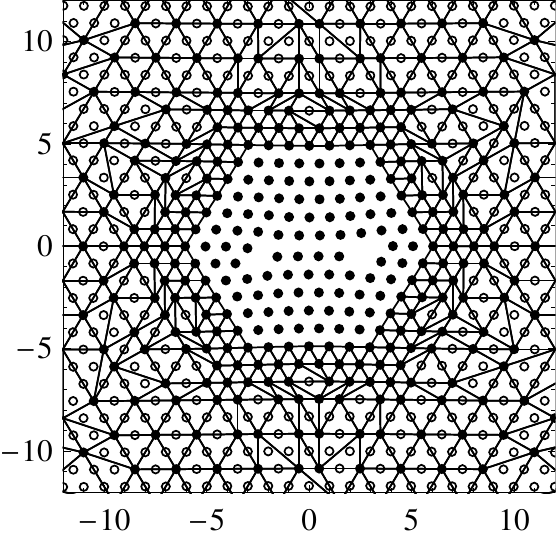}
\end{center}
\caption{A computed deformation of the atomistic system, closeup of the atomistic region.
The computation was done with the ECC method for $K=6$.
}
\label{fig:comparison_illustration:ecc6}
\end{figure}

In the numerical tests we compute the error of a local minimizer of the coupled A/C methods close to a precomputed minimizer of the exact atomistic energy.
More precisely, we first compute a local minimizer of the exact atomistic energy.
We then use it as (i) an initial guess for each computation with an A/C method, and (ii) as a reference solution to calculate the error of the A/C solution.
Of course, in ``real'' problems we cannot initialize a computation with the exact atomistic solution, as it is not known; we do this only to numerically study an approximate solution close to the exact one.

The absolute error of the numerical solution was calculated in the discrete $W^{1,\infty}$-norm.
For the test problem used in this work, the $W^{1,\infty}$-norm of the difference between the reference atomistic configuration (Figure \ref{fig:comparison_illustration:reference:closeup}) and the exact solution (Figure \ref{fig:comparison_illustration:ecc6}) was close to $1$; therefore the relative error is of the same magnitude as the absolute error.

A nonlinear conjugate gradient solver with line search \cite{Shewchuk1994} was used to find a stable equilibrium of an atomistic system.
A simple Laplace preconditioner was used to accelerate the convergence.

The discrete $W^{1,\infty}$-norm of the error of an approximate deformation was defined as follows: for each triangle formed by the neighboring atoms in the reference configuration we compute the Jacobian matrix of the mapping given by the approximate deformations of those atoms, take the difference to the corresponding Jacobian matrix of the exact solution, and compute the Frobenius norm of this difference.
Taking maximum over all triangles of neighboring atoms yields the discrete $W^{1,\infty}$-norm.

Three tests were conducted: the first one with the Lennard--Jones potential (section \ref{sec:num:comparison1}), the second one with a slowly decaying Morse potential (section \ref{sec:num:comparison2}), and the third one with the Lennard--Jones potential and a perturbed A/C interface (section \ref{sec:num:comparison3}).
The results of these computations are discussed in section \ref{sec:discussion}.

\subsection{Test with the Lennard--Jones potential}\label{sec:num:comparison1}

In the first test case we let atoms interact with the Lennard--Jones potential \eqref{eq:LJ_potential} with the cut-off distance $R=3.1$.
We computed the solutions for $5\le K\le 50$ with the two methods, ECC and QCE, and calculated their errors.

\begin{figure}
\begin{center}
\includegraphics{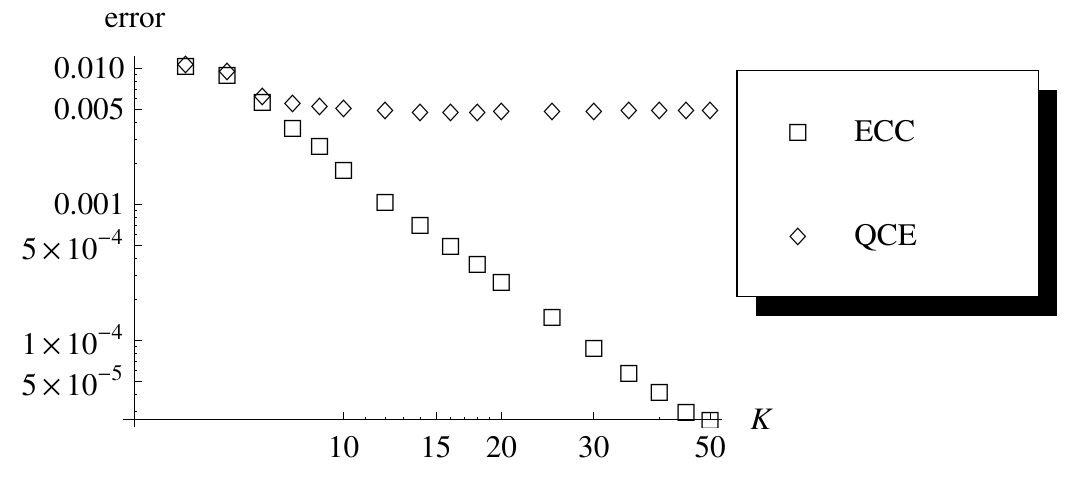}
\end{center}
\caption{$W^{1,\infty}$-errors of the computed solutions for the test with the Lennard--Jones potential.
	Computations were done by the two methods, ECC (cf.\ \eqref{eq:manyD:Eecc} or \eqref{eq:manyD:Eecc_effective_area}) and QCE \cite{TadmorOrtizPhillips1996}, for different sizes of the atomistic region ($5\le K\le 50$) and compared with the exact atomistic solution.
	For a small size of the atomistic region ($K=5,6,7$), both methods have a comparable error.
	However, for a larger size of the atomistic region ($K \gtrapprox 8$), the ECC method shows a steady convergence whereas the error of QCE remains $\calO(1)$ due to ghost forces.
}
\label{fig:comparison1}
\end{figure}

The results of computations are shown in Figure \ref{fig:comparison1}.
We can see that the ECC method does converge whereas QCE fails to converge due to ghost forces.
The error of QCE shows an initial decrease in error for small sizes of the atomistic region ($K=5,6,7$) but remains at a level of approximately $5\times 10^{-3}$ as $K$ is further increased.

The error magnitude of $5\times 10^{-3}$ is normally acceptable in the engineering applications.
Such a small error of QCE was due to the second nearest neighbor interaction being relatively weak compared to the nearest neighbor interaction for the Lennard--Jones potential.
In the next section we will see the results for the test case where the second nearest neighbor interaction is considerable.

\subsection{Test with a slowly decaying Morse potential}\label{sec:num:comparison2}

In this test case we chose the Morse potential \eqref{eq:Morse_potential} with $\alpha=3$ and the cut-off distance $R=3.1$.
The strength of such an interaction decays rather slowly, and the ghost force effects are more pronounced in this case.
The external compression $\mF = \begin{pmatrix} 0.97&0\\ 0&0.95 \end{pmatrix}$ was set as the boundary condition.
The exact solution for this test is similar to the one shown in Figure \ref{fig:comparison_illustration:ecc6}.

\begin{figure}
\begin{center}
\includegraphics{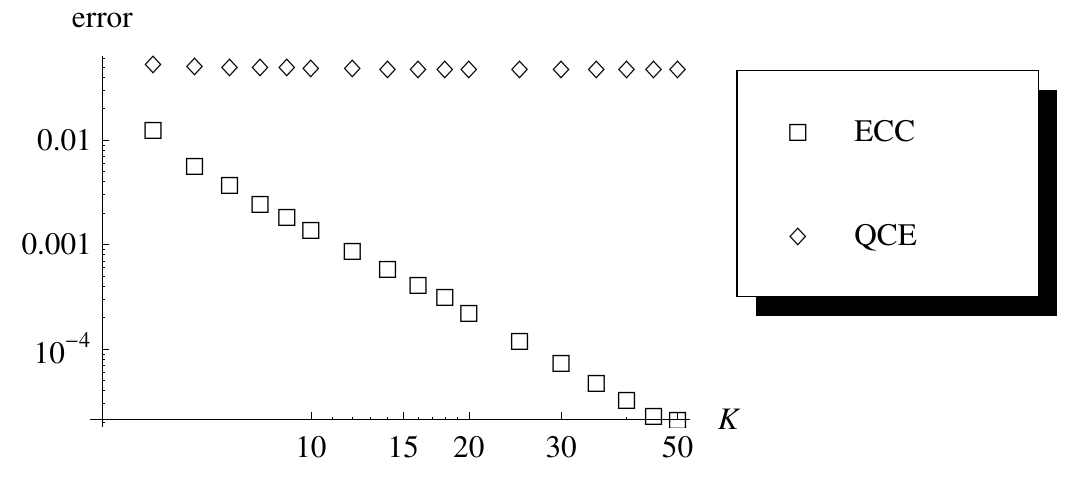}
\end{center}
\caption{$W^{1,\infty}$-errors of the computed solutions for the test with slowly decaying Morse potential.
	Computations were done by the two methods, ECC (cf.\ \eqref{eq:manyD:Eecc} or \eqref{eq:manyD:Eecc_effective_area}) and QCE \cite{TadmorOrtizPhillips1996}, for different sizes of the atomistic region ($5\le K\le 100$) and compared with the exact atomistic solution.
	The ECC method shows steady convergence as the size of the atomistic region increases.
	In contrast, the error of the QCE method is entirely dominated by the ghost force. 
}
\label{fig:comparison2}
\end{figure}

The results of computations are shown in Figure \ref{fig:comparison2}.
As can be seen, for such a slowly decaying interaction the QCE method exhibits a rather different behavior:
its error stayed at the level of $0.05$ no matter how large the atomistic region was.
The error of the ECC method, in contrast, steadily decayed with increasing $K$ and showed a behavior similar to the test case with the Lennard--Jones potential.

\subsection{Test with a nonaligned interface}\label{sec:num:comparison3}

In the last test the errors of the ECC and the ACC methods were compared for an aligned and a nonaligned A/C interface.
The test case is similar to the first one (section \ref{sec:num:comparison1}), except for a coarser mesh in the continuum region (see Figure \ref{fig:comparison3_illustration}).

\begin{figure}
\begin{center}
\hfill
\subfigure[Aligned A/C interface.]{\label{fig:comparison3_illustration:aligned}
	\includegraphics{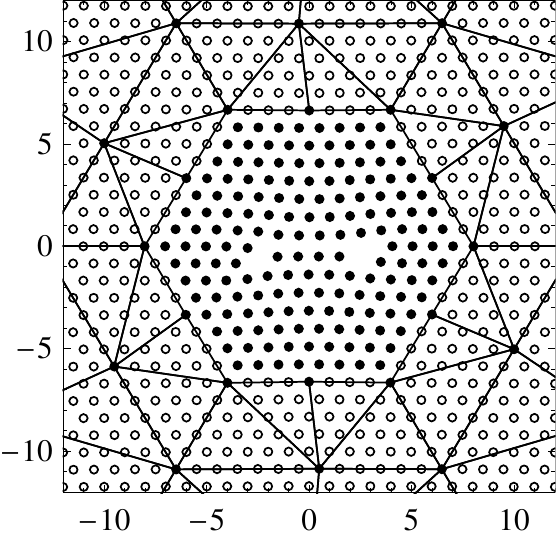}
}
\hfill \hfill
\subfigure[Nonaligned A/C interface.]{\label{fig:comparison3_illustration:nonaligned}
	\includegraphics{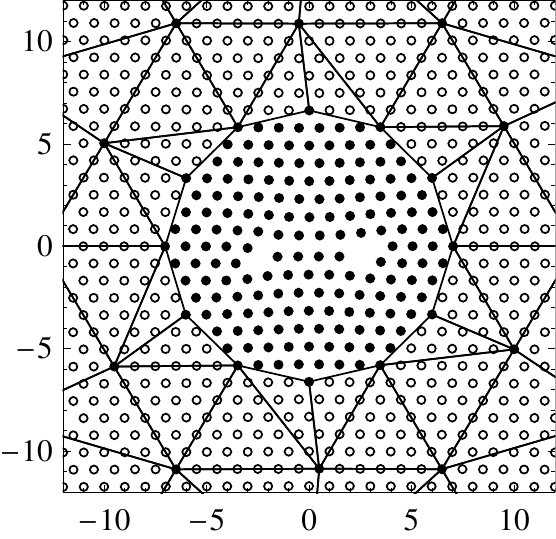}
}
\hfill $\mathstrut$
\end{center}
\caption{A/C interface of the test with the nonaligned interface (atomistic domain size $K=8$).
}
\label{fig:comparison3_illustration}
\end{figure}

\begin{figure}
\begin{center}
\includegraphics{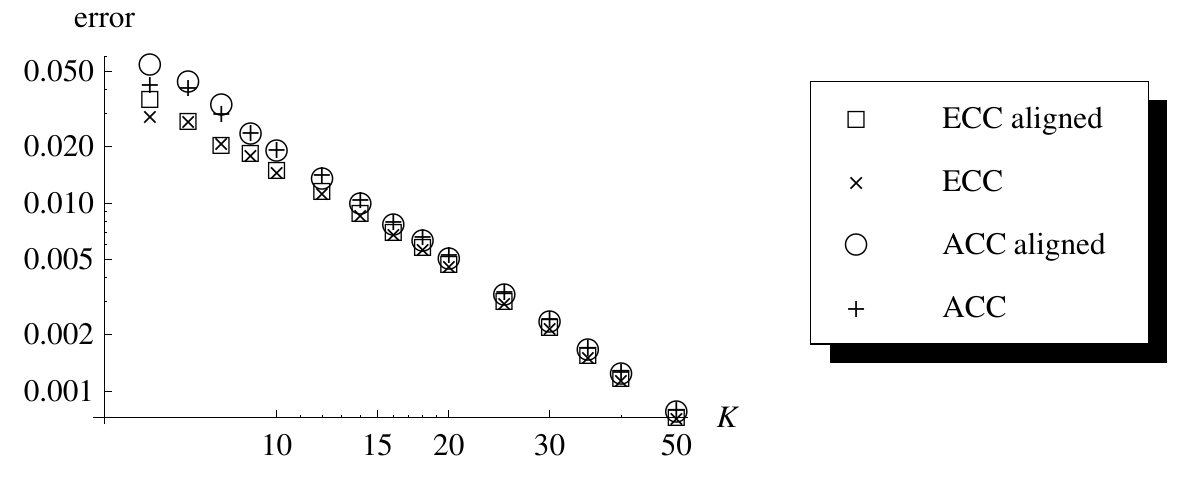}
\end{center}
\caption{$W^{1,\infty}$-errors of the computed solutions for the test with the Lennard--Jones potential and nonaligned interface.
	Computations were done by the two methods, ECC (cf.\ \eqref{eq:manyD:Eecc} or \eqref{eq:manyD:Eecc_effective_area}) and ACC (defined in \eqref{eq:manyD:Eacc}), for different sizes of the atomistic region ($6\le K\le 50$) and compared with the exact atomistic solution.
	The squares and circles correspond to the errors of ECC and ACC for the aligned interface, while the crosses and pluses correspond to the nonaligned interface.
	Both methods exhibit steady convergence as $K$ is increased, with almost no difference whether the interface was aligned or not.
}
\label{fig:comparison3}
\end{figure}

The results of computations are shown in Figure \ref{fig:comparison3}.
One can observe that for neither ECC nor for ACC the perturbation of the interface affected the error.
Also, the error of ECC was slightly less for small $K$ as compared to ACC, but for large $K$ the errors were very close.

\section{Discussion}\label{sec:discussion}

The main observation that can be made from the results of computations is that the proposed ECC and ACC methods exhibit steady convergence as the atomistic domain size is increased and are not sensitive to whether the nearest neighbor interaction dominate or whether the A/C interface is aligned with the lattice.
The solution by the QCE method, in contrast, does not converge due to ghost forces.
For a rapidly decaying potential, the error of QCE was within a practically acceptable limit but was found to be considerably larger for a slowly decaying interaction.

In the present paper we assumed that both the partition of the material into atomistic and continuum regions and the mesh in the continuum region are given.
In practice, a good choice of the regions and the mesh are often not known a priori, and one has to rely on certain algorithms to determine them adaptively (see, for instance, \cite{MillerTadmor2002, PrudhommeBaumanOden2006}).
However, rigorous a posteriori error bounds were proved only in 1D for consistent energy-based coupled methods  \cite{ArndtLuskin2008MMS, OrtnerSuli2008}.
The purely energy-based formulation of the proposed method and its consistency may help with deriving rigorous a posteriori error bounds in 2D.

Unfortunately, the bond formulation of the A/C energy developed in this work, according to Remark \ref{rem:no3D}, cannot be applied directly to couple the exact atomistic model with the Cauchy--Born continuum model in three dimensions (3D).
Nevertheless, it is possible to formulate an efficient numerical algorithm based on the bond formulation of the A/C energy in 3D \cite{Shapeev2011}.

The major challenge that remains to be addressed for the present method to become competitive with existing methods on realistic engineering problems is the formulation of the method for many-body potentials, as the two-body potentials are far from covering all existing atomistic models.
In this regard, we should note that there exists a method of A/C coupling (namely, the geometrically consistent scheme of E, Lu, and Yang \cite{ELuYang2006}) which is free from ghost forces under the restriction of planar interface with no corners but for general potentials.

It would also be interesting to extend the present method to coupling an atomistic model with a nondiscretized continuum model in 2D and 3D.
In that case one can discretize the continuum model with higher-degree polynomials and obtain an increased accuracy of the overall method.
Another extension which would be useful for applications is to formulate the proposed method for complex lattices (for a possible application, see \cite{Kastner2003} for two-dimensional models of metallic alloys described in terms of two-body potentials).

\section{Conclusion} \label{sec:conclusion}

We considered the problem of consistent coupling of atomistic and continuum models of materials, limited to the case of a two-body potential and one or two spatial dimensions.
We proposed two versions of energy-based coupling which are consistent (i.e., do not suffer from ghost forces).
The coupling is based on judiciously defining the contributions of the atomistic bonds to the discrete and the continuum potential energies.
The same coupling in 1D has been independently proposed and analyzed in \cite{LiLuskin}.
The latest works conducted on the proposed coupling include its numerical analysis in 2D \cite{OrtnerShapeev2011} and extension to 3D \cite{Shapeev2011}.

\section*{Acknowledgments} 
The author is indebted to Christoph Ortner for motivating discussions during the OxMOS/MD-network workshop and for many valuable comments and suggestions on the manuscript.
The author is thankful to an anonymous referee for the useful comments that helped to improve this paper.

\end{document}